\documentclass[11pt,reqno,oneside]{amsart}
\usepackage{amsmath}
\usepackage{amssymb}
\usepackage{graphicx}
\DeclareGraphicsExtensions{.eps}
\usepackage{caption} 
\usepackage{subcaption}
\usepackage{multirow}
\usepackage{placeins}
\usepackage{amssymb}
\usepackage{amsbsy}
\usepackage{longtable}
\usepackage{comment}
\usepackage{accents}
\usepackage{hyperref}

\numberwithin{equation}{section}

\newtheorem{thm}{Theorem}[section]

\newtheorem{lemma}[thm]{Lemma}

\newtheorem{coro}[thm]{Corollary}

\newtheorem{rmk}[thm]{Remark}

\newcommand{\R}{\mathbb{R}}
\newcommand{\N}{\mathbb{N}}

\newcommand{\cP}{\mathcal{P}}

\newcommand{\cT}{\mathcal{T}}

\newcommand{\gve}{\varepsilon}

\newcommand{\nempty}{\neq\emptyset}
\newcommand{\wt}{\widetilde}
\newcommand{\wh}{\widehat}

\newcommand{\dx}{\textrm{dx}}
\newcommand{\dt}{\textrm{dt}}
\newcommand{\h}{\vec{h}}

\def\XXint#1#2#3{{\setbox0=\hbox{$#1{#2#3}{\int}$}
\vcenter{\hbox{$#2#3$}}\kern-.5\wd0}}

\begin{document}
\author[Ignacio Ojea]{Ignacio Ojea}

 \title{Anisotropic finite elements for elliptic problems with singular data}
 \address{Departamento de Matem\'atica Facultad de Ciencias Exactas y Naturales, Universidad de Buenos Aires, - Inst. de Investigaciones Matemáticas ``Luis A. Santal\'o'', CONICET-UBA. Partially supported by a fellowship of CONICET.} 
 \email{iojea@dm.uba.ar}  

\begin{abstract}
We study the problem $-\Delta u = \gamma$, where $\gamma$ is a singular measure, with support on a curve or a point. We prove that optimal rates of convergence for the finite element method can be obtained using properly graded meshes. In particular, we consider isotropic graded meshes when $\gamma$ is a point Dirac delta, and anisotropic graded meshes when $\gamma$ is a measure supported on a segment. Numerical experiments are shown that verify our results, and lead to interesting observations.
\end{abstract}

\maketitle

\section{Introduction}
In this paper we study the Poisson equation with data given by a singular measure. In the simplest case, such data will be a Dirac delta distribution supported on a point. More generally, we are interested in data given by a finite measure with support on a curve. Our model problem is:
\begin{equation}\label{problem}
\left\{\begin{array}{cl} -\Delta u(x) = \gamma(x) & x\in\Omega\subset\mathbb{R}^n\\
u(x) = 0 & x\in\partial\Omega,
\end{array}\right.
\end{equation}
where $n=2,3$, $\Omega$ is a smooth convex domain, and $\gamma$ is a finite measure with support on a curve $\Gamma\subset\Omega$. Particularly, if $\phi:[-L,L]\to\R^n$ is an arc-length parametrization of $\Gamma$, there is a function $\hat{\gamma}:[-L,L]\to\mathbb{R}$, so that the measure $\gamma$ is given by:
$$\int_{\Omega} f(x)\gamma(x) \dx = \int_{-L}^{L} f(\phi(t))\hat{\gamma}(t) \dt,$$
for every $f\in C^{\infty}_0(\Omega)$. We assume $\hat{\gamma}\in C^{1}([-L,L])$. We are particularly interested in the simple case where $\Gamma$ is a straight line, where the anisotriopic behaviour of the solution of \eqref{problem} is easier to understand.

This kind of problem arises in many contexts. The case of point singularities can model, for example, point sources in electromagnetic or convection-diffusion problems. It has been largely studied. We can mention, for example, \cite{B_delta} and \cite{S_delta} where a priori estimates are given for the norm of the error measured in $L^2$ and in fractional Sobolev spaces $H^s$ for some $0<s<1$.
In \cite{ABSV_delta} an a priori analysis is carried out on weighted Sobolev spaces, and the $L^2$ norm of the error is bounded. A posteriori estimates are given in \cite{ABR_delta} in $L^p$, with $1<p<\infty$ and in $W^{1,p}$ for $1<p<2$. However, the arguments presented there hold only for $n=2$. In \cite{KW_delta} quasi-uniform meshes are used, obtaining quasi-optimal rates of convergence for finite element methods of order one, and optimal rates of convergence for higher order methods. However, the authors only consider local error norms: the error is measured on a domain excluding the singularity. Similar results are obtained in \cite{KVW_deltaline} where local error estimates are proved for the Poisson equation with a source given by a Dirac measure supported on a curve.
Finally \cite{AGM_delta} proposes a posteriori error estimates for weighted Sobolev spaces, where the possible weights are given by powers of the distance to the singular point, both for $n=2$ and $n=3$. On the other hand, singularities supported on a curve are used in \cite{DQ_1d3d,D_1d3d} to model two coupled diffusion-reaction problems, one on the curve, and one in the domain.
The goal of those papers is to study the flow of blood through tissues: the domain represents a mass of tissue whereas large blood vessels are described by curves. However, the same setting can be used to model fluid flow in three dimensional porous media with fractures represented by one-dimensional subsets. In \cite{D_1d3d} a priori graded meshes are used to solve the problem using the finite element method, and estimates are found for a weighted norm of the error. In fact, the weighted analysis of the point-singular problem given in \cite{AGM_delta} follows closely the arguments introduced in \cite{DQ_1d3d} and \cite{D_1d3d}. Other applications of problem \eqref{problem} can be seen in \cite{ZZ_DeltaCurve} and the references therein.

As it is pointed out in \cite{AGM_delta}, weighted Sobolev spaces as the ones used in \cite{DQ_1d3d,D_1d3d} (but also in \cite{ABSV_delta}) seem to be more appropriate than $W^{1,p}$ spaces (used in \cite{ABR_delta}) or $H^s$ spaces (considered in \cite{B_delta,S_delta}), since the norm of the weighted spaces is only weakened near the singularity, and not in the whole domain $\Omega$.

Here we are interested in problem \eqref{problem} as a model for heat diffusion produced by the heating of gold nanoparticles through laser beams (see, for example \cite{Petal_nano}). Spherical nanoparticles can be represented as point-sources, whereas ``elliptic'' nanoparticles or arrays of nanoparticles can be represented by one-dimensional singularities. In this context, it is of particular interest the case in which $\Gamma$ is a segment.

In the first sections of the paper we state the general setting for the problem in weighted spaces of Kondratiev type and recall regularity results proved in \cite{O_regularity}. In Section \ref{S:discreteproblem} we give a general setting for the discrete problem and prove a weighted version of Aubin-Nitsche's Lemma. Section \ref{S:pointfem} is devoted to a priori estimates for data given by a point Dirac delta where the main ideas are easier to understand. Section  \ref{S:segmentfem} treats the case of measure supported on a segment, giving a priori estimates for both isotropic and anisotropic graded meshes. We focus on the case $n=3$. The case $n=2$ is commented later. Finally, Section \ref{S:experiments} show numerical experiments and its results.

\section{Notation and Preliminaries}\label{S:notation}

We consider a smooth convex domain $\Omega\subset\mathbb{R}^n$, $n=2,3$. $\Gamma\subset\Omega$ is a curve such that $\Gamma\cap\partial\Omega=\emptyset$.
 We denote $m$ the dimension of $\Gamma$, being $m=1$ for a curve, and $m=0$ for a point.

 Let $r(x)=d(x,\Gamma)$. We define a neighborhood of $\Gamma$ by:
 $$B(\Gamma,\gve) = \{x\in\R^n:\; r(x)<\gve\}.$$

 For $m=1$, we focus on the anisotropic behaviour of the solution. For this, we consider the case in which $\Gamma$ is a straight line given by:
 \begin{equation}\label{segment}
 \Gamma = \{(0',x_n)\in\R^{n-1}\times\R:\; -\dfrac{L}{2}\le x_n\le \dfrac{L}{2}\}.
 \end{equation}
 where $0'$ detones the null vector in $\R^{n-1}$.
 Since $\Gamma\cap\partial\Omega=\emptyset$, there is some $\rho_0>0$ such that $B(\Sigma,\rho_0)\subset\Omega$. For convinience, we assume $\rho_0=1$.

 It is sometimes useful to decompose the neighbourhood $B(\Gamma,\rho)$ of $\Gamma$ in its cylindrical part:
 $$C(\Gamma,\rho) = \{(x',x)\in\R^n:\; -L<x_n<L,\, \|x'\|<\rho\},$$
 and the extreme semi-balls:
 $$B^{out}(\pm L,\rho) = \{x\in\R^n:\; \|x-(0',\pm L)\|<\rho\}.$$

 We write $B^{out}(\Gamma,\rho) = B^{out}(-L,\rho)\cup B^{out}(L,\rho)$.

Finally, it will be necessary to consider also the distance to the extreme points of $\Gamma$, $\{(0',-L),(0',L)\}$, that we denote $r_e(x)$.

 We use the standard notation for derivatives: $\alpha=(\alpha_1,\dots,\alpha_n)\in\N_0^n$ is a multiindex and $D^\alpha u$ stands for $\partial^{\alpha_1}_{x_1}\dots\partial^{\alpha_n}_{x_n}u$. $|\alpha|=\alpha_1+\dots+\alpha_n$. Sometimes we take $\alpha'$ such that $\alpha=(\alpha',\alpha_n)$.
 Moreover, if $\vec{h}\in\R^n$, $\vec{h}^\alpha$ stands for $h_1^{\alpha_1}\dots h_n^{\alpha_n}$.

 We denote with $C$ a generic constant that may change from line to line. We also write $a\lesssim b$ whenever $a\le C b$ for some constant $C$ independent of $a$ and $b$. We say $a\sim b$ when $b\lesssim a\lesssim b$.






\section{Weak formulation and regularity}\label{S:weakformulation}

\subsection{Weak Formulation}
The weak formulation for \eqref{problem} consists in finding $u\in V$ such that:
\begin{equation}\label{weakproblem}
\begin{array}{cl}
\int_\Omega \nabla u\nabla v = \int_\Omega \gamma v & \forall v\in V',
\end{array}
\end{equation}
where the spaces $V$ and $V'$ that give sense to the so far formal statement \eqref{weakproblem} are to be defined.
Since the source term $\gamma$ does not belong to the dual space of $H^1_0(\Omega)$ (except when $n=2$, $m=1$), it is not possible to use the usual test and ansatz space $V'=V=H^1_0$. We consider weighted Sobolev spaces.

Let $L^2_\sigma(\Omega)$ be the space of measurable functions $v$ with: $\|v\|_{L^2_\sigma(\Omega)}:=\|vr^\sigma\|_{L^2(\Omega)}<\infty$. We also define $H^1_\sigma(\Omega)$ the space of functions $v\in L^2_{\sigma}(\Omega)$ with derivatives in $L^2_\sigma(\Omega)$.
The dual space for $L^2_{\sigma}$ is $L^2_{-\sigma}$, with respect to the duality product:
$\langle u, v\rangle = \int_{\Omega} u(x) v(x) \dx$.
Observe that, for a bounded domain $\Omega$, we have that $\|v\|_{L^2_{\sigma+\varepsilon}(\Omega)}\le \|v\|_{L^2_{\sigma}(\Omega)}$,  $\forall\varepsilon>0$, and the continuous inclusion $L^2_{\sigma}\subset L^2_{\sigma+\varepsilon}$.

Many results on Sobolev spaces $W^{k,p}$ can be extended to weighted Sobolev spaces $W^{k,p}_{\omega}$ when the weight $\omega$ belongs to the Muckenhoupt class $A_p$  (see \cite{M_Ap}). The following result gives a characterization of the weights of the form $r(x)^\sigma$ that belong to the class $A_2$ (see \cite[Lemma 3.3]{DLG_Stokes}):

\begin{lemma}\label{L:DLG_Ap}
Let $F\subset\R^n$ be a compact set of dimension $m$, and let $r(x)$ be the distance from $x$ to $F$. If
\begin{equation}\label{A2cond}
  -\frac{n-m}{2}<\sigma<\frac{n-m}{2},
\end{equation}
then $r(x)^{2\sigma}$ belongs to the class $A_2$.
\end{lemma}

A consequence of this lemma is that, for $\sigma$ satisfying \eqref{A2cond} and $v\in H^1_\sigma(\Omega)$ the weighted Poincar\'e inequality
$$\|v\|_{L^2_{\sigma}(\Omega)} \le C_P \|\nabla v\|_{L^2_{\sigma}(\Omega)}$$
holds whenever $v$ has mean value zero on $\Omega$, or support in $\Omega$. This allows D'Angelo \cite{D_1d3d} to consider the space $W_\sigma = \{v\in H^1_{\sigma}: \, v|_{\partial\Omega}= 0\}$, with the norm:
$$\|v\|_{W_{\sigma}} = \|\nabla v\|_{L^2_{\sigma}},$$
which is equivalent to the $H^1_{\sigma}$ norm.

The goal is to complete the definition of our weak problem \eqref{weakproblem} taking $V=W_\sigma$ and $V'=W_{-\sigma}$.
Hence, we consider the more general problem that reads as: \emph{Given $f\in W_{-\sigma}'$, find $u\in W_{\sigma}$ such that: }
\begin{equation}\label{modelproblem}
  \langle \nabla u,\nabla v\rangle = \langle f,v\rangle, \quad \forall v\in W_{-\sigma}.
\end{equation}
The existence and uniqueness of solution of the problem thus set is proved in \cite{D_1d3d} for the case $m=1$, $n=3$, and in \cite{AGM_delta} for $m=0$ and $n=2,3$, provided that $\sigma$ satisfies \eqref{A2cond}.
It is easy to prove the same result stands when $n=2$, $m=1$.

The only remaining issue is to prove that our right hand side $\gamma$ belongs to $W_{-\sigma}'$. This has also been proven, in \cite{DQ_1d3d} ($n=3$, $m=1$) and in   \cite{AGM_delta} ($m=0$), and holds true whenever
\begin{equation}\label{singularwelldefcond}
  \frac{n-m}{2}-1<\sigma<\frac{n-m}{2}.
\end{equation}

In conclusion, we have that \eqref{weakproblem} has a unique solution $u\in W_{\sigma}$ for $\sigma$ satisfying \eqref{singularwelldefcond}. And thanks to the Poincar\'e inequality mentioned above, $u$ belongs to $H^1_{\sigma}$.

\subsection{Regularity in Kondratiev type spaces}

However in order to obtain a priori error estimates, we need information about the regularity of $u$ up to its second order derivatives.
Such regularity can be more accurately expressed in terms of Kondratiev spaces, rather than standard Sobolev spaces. Hence, we introduce the Kondratiev type space $K^\ell_\eta(\Omega)$, formed by all functions $v$ having weak derivatives of order $\alpha$, for $0\le|\alpha|\le \ell$, equipped with the norm:
$$\|u\|_{K^\ell_{\eta}}^2 = \sum_{|\alpha|\le \ell} \int |D^{\alpha}u(x)|^2r(x)^{2(\eta+|\alpha|)} dx $$

The following result, proved in \cite{O_regularity} gives a characterization of the solution $u$ of problem \eqref{problem}:

\begin{thm}\label{T:norms}
Let $u$ be the solution of problem \eqref{problem}, then, $u\in K^2_{\sigma-1}(\Omega)$, for every  $\sigma>\frac{n-m}{2}-1$; except for the particular case $n=2$, $m=1$, where $u\in L^2_{\sigma}(\Omega)$ and $\nabla u\in K^1_{\sigma}(\Omega)$, provided that $\sigma>-\frac{1}{2}$.
\end{thm}

For $m=0$ (point Dirac delta), this result follows directly from the study of the fundamental solution for the Laplacian. For $m=1$, a much more complicated analysis is needed. We refer the reader to \cite{O_regularity} for a complete proof.
We want to remark, however, that in the case $m=1$, $n=3$ this result is consistent with the regularity of $u$ assumed in \cite{D_1d3d}.

For the case in which $\Gamma$ is the segment \eqref{segment}, Theorem \ref{T:norms} can be refined, detailing the anisotropic behaviour of the derivatives of $u$. Such a refinement is given in the following theorem, also proved in \cite{O_regularity}. We recall that $r_e(x)$ stands for the distance from $x$ to the extreme points of $\Gamma$.

\begin{thm}\label{T:anisotropic}
Let $\Gamma$ be the segment \eqref{segment}, and $u$ the solution of \eqref{problem}. Consider a compact set $S\subset C(\Gamma,1)$, and take $r_{S,e} = min_{x\in S} r_e(x)$. If $r_{S,e}> 0$, then there is a constant $C$ depending on the measure $|S|$ of $S$ such that:
\begin{equation}\label{anisotropicDu}
r_{S,e}^{\frac{1}{2}} \|D^\alpha \partial_{x_n} u\|_{L^2_{\eta}(S)}\le C   \quad \textrm{ for every } \eta>\frac{n-1}{2}-2+|\alpha|
\end{equation}
for $|\alpha|=0,1$ in the case $n=3$ and for $|\alpha|=1$ in the case $n=2$. For the second derivative with respect to $x_n$, in $\R^3$, we have the better estimate:
\begin{equation}\label{anisotropicDupar}
r_{S,e}^{\frac{3}{2}}\|\partial^2_{x_n} u\|_{L^2_{\eta}(S)} \le C \quad \textrm{ for every } \eta>-\frac{n-1}{2}
\end{equation}
\end{thm}

The goal of this result is twofold: first, to show that the derivatives of $u$ along the direction parallel to $\Gamma$ are actually smoother than the ones along directions orthogonal to $\Gamma$. Second, to show that, though smoother, the derivatives along the direction parallel to $\Gamma$ lose regularity near the extreme points of the segment, and to make explicit how this loss of regularity depends on the distance $r_e$.
In Section \ref{S:segmentfem}, the compact sets $S$ will be the elements of the finite element mesh.

\subsection{An improved Poincar\'e inequatlity}

Dealing with Kondratiev spaces, we will use the following improved Poincar\'e inequality:

\begin{thm}\label{T:impPoincare}
  Let $\Omega\subset\R^n$, $\Gamma\subset\Omega$ a curve $(m=1)$ or point $(m=0)$ and $v\in W_{\sigma}$. If
  \begin{equation}\label{impPoincarecond}
    -\frac{n-m}{2}+1<\sigma<\frac{n-m}{2}.
    \end{equation}
  Then, there is a constant $C_P$ independent of $v$ such that:
  \begin{equation}\label{impPoincare}
    \|v\|_{L^2_{\sigma-1}(\Omega)}\le C_P \|\nabla v\|_{L^2_{\sigma}(\Omega)}.
  \end{equation}
\end{thm}

We give a sketch of the proof of \eqref{impPoincare}. For that, let us begin by recalling the following theorem (see \cite[Theorem 1]{SW_fracint}).
\begin{thm}\label{T:SW}
   Let $I_\alpha$ be the fractional integral: $I_\alpha(f)(x)=\int |x-y|^{\alpha-n}f(y) dy,$ and $w$ and $v$ be positive weights. If for some $r>1$ there is a constant $C_r$ such that,
\begin{equation}\label{SWcond}
  |Q|^\frac{\alpha}{n}\left(\frac{1}{|Q|}\int_Q w^r\right)^\frac{1}{2r}\left(\frac{1}{|Q|}\int_Q v^{-r}\right)^{\frac{1}{2r}}< C_r, \quad \textrm{for every cube }Q\subset \R^n
\end{equation}
Then, the inequality:
 $$\int |I_\alpha(f)(x)|^2 w(x) \dx \le C \int|\nabla f(x)|^2 v(x) \dx$$
 holds for every $f\geq 0$.
  \end{thm}

We want to apply this result taking $w=r^\beta$ and $v=r^\sigma$. The following lemma indicates how should $\beta$ and $\sigma$ be taken. It can be easily proven exactly as Lemma \ref{L:DLG_Ap}, so we refer the reader to \cite[Lemma 3.3]{DLG_Stokes}:
\begin{lemma} Let $F\subset \R^n$ be a compact set of dimension $m$, and let $r(x)$ denote the distance from $x$ to $F$. Then, if we take:
  $$\sigma = \beta+\alpha,\quad \beta>-\frac{n-m}{2},\quad \sigma<\frac{n-m}{2},$$
  the weights $w=r^\beta$ and $v=r^\sigma$ satisfy \eqref{SWcond}.
\end{lemma}

Finally, we are able to prove the Theorem:

\begin{proof}[Proof of Theorem \ref{T:impPoincare}]
  It is a well known fact that given a function $f\in C^1$, with support or mean value zero on a cube $Q$, the inequality $|f(x)|\le C|I_1(\nabla f)(x)|$ holds for every $x\in Q$. Since $v|_{\partial\Omega}=0$, it can be extended by zero to a cube $Q\supset\Omega$.
  Taking $f=\chi_Q |\nabla v|$, where $\chi_Q$ is the characteristic of $Q$, we have $\|v\|_{L^2_{\sigma-1}(\Omega)}\le C\|I_1(\nabla v)\|_{L^2_{\sigma-1}(\Omega)}$.
  We apply Theorem \ref{T:SW} with $\alpha = 1$, $w=r^{\sigma-1}$ and $v=r^{\sigma}$ and the result follows. The proof is completed by a density argument.
\end{proof}

\begin{rmk}
Observe that \eqref{singularwelldefcond} implies \eqref{impPoincarecond}, except for the case $n=2$ and $m=1$, where \eqref{impPoincarecond} gives and empty range.
\end{rmk}



\section{The discrete problem}\label{S:discreteproblem}

In this section we discuss some general aspects of the discrete version of our problem.
Let $\mathcal{T}_h$ be a triangulation of $\Omega$. For every $T\in\cT_h$ we take:
$$r_T=d(T,\Gamma), \quad \overline{r}_T=\sup_{x\in T} d(x,\Gamma), \quad h_T=diam(T).$$
$S_T$ denotes the patch of elements adjacents to $T$: $S_T = \{T'\in\cT_h:\; \overline{T'}\cap\overline{T}\nempty\}.$

We distinguish two classes of elements:
$$\cT^{near}_h = \Big\{T\in\cT_h:\; \overline{S_T}\cap\Gamma\nempty\Big\},\quad \cT^{far}_h = \cT_h\setminus \cT^{near}_h.$$
Sometimes with a little abuse of notation we use $\cT^{near}_h$ and $\cT^{far}_h$ to denote the regions $\cup_{T\in\cT^{near}_h}T$ and $\cup_{T\in\cT^{far}_h}T$, respectively. For $T\in\cT_h^{near}$, we denote:
$S_T' = \{T'\in\cT^{far}_h:\; \bar{T}\cap\bar{T}'\neq\emptyset\}.$

We define the weighted discrete space:
$$W_h = \{u_h\in C(\Omega):\; u_h|_T\in \cP_1(T)\,\forall T\in\cT_h,\; u_h|_{\partial\Omega} = 0\},$$
equipped with the weighted norm:
$$\|u_h\|_{W_h}^2=\sum_{T\in\cT_h} \overline{r_T}^{2\sigma}\|u_h\|_{L^2(T)}^2.$$

The well posedness of the model problem \eqref{modelproblem} in $W_h$ as well as its stability is proved in \cite{D_1d3d} for $m=1$, $n=3$ and in \cite{AGM_delta} for $m=0$. As a conclusion we have that for $\sigma$ satisfying \eqref{A2cond} the optimal estimate for the Galerkin approximation
\begin{equation}\label{stability}
  \|u-u_h\|_{W_{\sigma}}\le C \inf_{v_h\in W_h}\|u-v_h\|_{W_{\sigma}}
\end{equation}
holds. See \cite[Section 3]{AGM_delta} and \cite[Theorem 3.4]{D_1d3d}.

Therefore, in order to obtain estimates for the $W_\sigma$ norm of the discretization error it is enough to prove estimates for the $W_\sigma$ norm of the interpolation error.  However, in order to obtain estimates in $L^2_\beta$, we need a weighted version of Aubin-Nitsche's lemma. It is particularly interesting the case of the standard $L^2$ norm, given by $\beta=0$.

Let us take $g_{\beta} = r^{2\beta}|u-u_h|$, then, we have $\|g_{\beta}\|_{L^2_{-\beta}(\Omega)}=\|u-u_h\|_{L^2_{\beta}(\Omega)}$. We want to estimate:
$$\|u-u_h\|^2_{L^2_{\beta}(\Omega)} = \langle u-u_h,g_{\beta}\rangle$$
For that, we follow the classical argument of Aubin-Nitsche's lemma setting the problem of finding $\varphi_\beta\in W_{-\sigma}$ such that:
 \begin{equation}\label{adjoint}
   \langle \nabla\varphi_\beta,\nabla w\rangle = \langle g_{\beta},w\rangle, \quad\forall w\in W_{\sigma}.
\end{equation}
Observe that if $\sigma\le \beta+1$, then $g_{\beta}\in L^2_{-\beta}\subset L^2_{1-\sigma} = (L^2_{\sigma-1})'$, and, thanks to \eqref{impPoincare}, $W_\sigma\subset L^2_{\sigma-1}$.
Hence, the map $g_{\beta}:W_{\sigma}\to\R$ given by $g_{\beta}(w) = \langle g_{\beta},w\rangle$ is a linear continuous map.
Moreover, since \eqref{singularwelldefcond} holds, $-\sigma$ satisfies \eqref{A2cond}, and the adjoint problem \eqref{adjoint} is well posed, and admits a unique solution $\varphi_\beta\in W_{-\sigma}$.

We continue as in the unweighted case, since $u-u_h\in W_{\sigma}$:
$$\langle \nabla \varphi_\beta, \nabla (u-u_h)\rangle = \langle g_\beta,u-u_h\rangle,$$
whereas, for every $\varphi_h\in W_h$: $\langle \nabla \varphi_h,\nabla (u-u_h) \rangle = 0,$
which gives:
$$\langle g_\beta,u-u_h \rangle = \langle \nabla (u-u_h), \nabla(\varphi_h-\varphi_h)\rangle \le \|u-u_h\|_{W_{\sigma}}\|\varphi_\beta-\varphi_h\|_{W_{-\sigma}},$$
and we have completed the proof of our weighted Aubin-Nitsche's lemma:
\begin{lemma}[Weighted Aubin-Nitche's Lemma]\label{L:aubinnitsche}
  Given $\sigma\le\beta+1$ satisfying \eqref{singularwelldefcond}, we have:
  $$\|u-u_h\|_{L^2_{\beta}(\Omega)}\le
  \|u-u_h\|_{W_{\sigma}}\inf_{\varphi_h\in W_h}\|\varphi_\beta-\varphi_h\|_{W_{-\sigma}}$$
  where $\varphi_\beta$ is the solution of \eqref{adjoint}.
\end{lemma}

The weighted regularity of the solution of Poisson's equation for weights in the $A_p$ class is proven in \cite{DST_ApPoisson}, so we have that:
$$\|\varphi_\beta\|_{H^2_{-\beta}(\Omega)}\lesssim \|g_\beta\|_{L^2_{-\beta}(\Omega)}=\|u-u_h\|_{L^2_{\beta}(\Omega)}.$$
 Thence, estimates for the discretization error in $L^2_\beta$ will follow from estimates of the interpolation error for $u$ in $W_{\sigma}$, and of the interpolation error for $\varphi_\beta$ in $W_{-\sigma}.$

For the rest of the paper we use $\sigma$ to denote an exponent satisfying \eqref{singularwelldefcond}, so $u\in W_\sigma$.
We use $\beta\geq\sigma-1$ to denote an exponent corresponding to a space $L^2_\beta$ where $u$ belongs.  For convinience, we also assume $\beta>\frac{n-m}{4}-1$ for $m=0$ or $n=3$, $m=1$ and $\beta>\frac{1}{4}$ for $m=1$, $n=2$.
This guarantees that $g_\beta\in L^2$, and $\varphi_\beta\in H^2$, which allows us to use a standard interpolator for estimating the error for problem \eqref{adjoint}.

\section{Isotropic graded meshes for point-Delta sources}\label{S:pointfem}

In the case of a point Dirac delta ($m=0$, $\Gamma=\{{\bf 0}\}$) we take an isotropically graded mesh $\cT_h$, according to the rule:
$$
h_T \sim \left\{\begin{array}{cl}
h^{\frac{1}{\mu}} & \textrm{ if } {\bf 0}\in\bar{T}, \\
h r_T^{1-\mu} & \textrm{ if } 0<d(T,{\bf 0})\le 1, \\
h & \textrm{ if } 1< d(T,{\bf 0}),
\end{array} \right.
$$
for some $\mu$ to be determined. We assume that the origin of coordinates ${\bf 0}$ where the singularity lies is one of the vertices of $\cT_h$. We divide the analysis in two parts, studying first the error in $W_{\sigma}$, and afterwards the error in $L^2_{\beta}$.

\subsection{Estimates in $W_{\sigma}$}
Since $u$ does not belong to $H^1$ near the singularity, we need to introduce a suitable interpolation operator $I_h: K_{\sigma-1}^2\to W_h$, for $\sigma>\frac{n}{2}-1$. Let $\{x_i\}_{i=1}^{N_h}$ be the set of nodes of $\cT_h$, and $\{\phi_i\}$ the nodal basis: $\phi_i|_T\in\cP_1$, $\phi_i(x_j)=\delta_{ij}$. We define:
\begin{equation}\label{Ihdef}
  I_h u (x) = \sum_{i=1}^{N_h} a_i \phi_i(x).
\end{equation}
For nodes $x_i\in\cT^{far}_h$, $I_h$ is the Lagrange interpolator, given by $a_i = u(x_i).$
This definition is allowed by the fact that functions in $K^2_{\sigma-1}$ belong to $H^2(\cT^{far}_h)$. On the other hand, for nodes  $x_i\in\cT_h^{near}$ we take $a_i = 0.$

The following lemma was proved in \cite[Lemma 3.5]{D_1d3d}.

\begin{lemma}\label{L:cotasIh}
The interpolator $I_h$ satisfies the following properties:
\begin{equation}\label{Ihfar}
|u-I_h u|_{H^\ell(T)}\le C h_T^{2-\ell}|u|_{H^2(T)},\quad \textrm{if }T\in \cT^{far}_h, \,\ell=0,1
\end{equation}
\begin{equation}\label{Ihnear}
|I_h u|_{K^1_{\sigma-1}(T)}\le C_\sigma \|u\|_{K^2_{\sigma-1}(S_T')},\quad \textrm{if }T\in \cT_h^{near},
\end{equation}
for every $u\in K^2_{\sigma-1}$ and $\ell=0,1$. Moreover, $C_\sigma \sim (2\sigma+n)^{-\frac{1}{2}}$.
\end{lemma}

Inequality \eqref{Ihfar} is due to the fact that $I_h|_T$ is the Lagrange interpolator $\forall T\in\cT^{far}$. The stability estimate \eqref{Ihnear} is very similar to \eqref{Pihstability}, that is proven later for the classical Lagrange interpolator. With this, we can prove:

\begin{thm}\label{T:approxIh}
Let $\sigma$ satisfy \eqref{singularwelldefcond} and take $\mu\le\sigma-1-\eta$, then:
\begin{equation}\label{approxIh}
|u-I_h u|_{K^1_{\sigma-1}(\Omega)} \le C_{\sigma} h \|u\|_{K^2_{\eta}(\Omega)}.
\end{equation}
\end{thm}
\begin{proof}
We prove the result elementwise. For every $T$ such that $d(T,{\bf 0})>1$, the result follows directly from \eqref{Ihfar}. For $T\in\cT^{far}_h$, such that $0<d(T,{\bf 0})\le 1$:
\begin{align*}
|u-I_h u|_{K^1_{\sigma-1}(T)} &\le \bar{r}_{T}^{\sigma} |u-I_h u|_{H^1(T)} \lesssim \bar{r}_T^{\sigma} h_T|u|_{H^2(T)} \lesssim \bar{r}_T^{\sigma} h_Tr_T^{-\eta-2}|u|_{K^2_{\eta}(T)} \\
&\lesssim r_T^{\sigma-\eta-2} hr_T^{1-\mu} |u|_{K^2_{\eta}(T)} \lesssim h r_T^{\sigma-1-\eta-\mu} |u|_{K^2_{\eta}(T)}\lesssim h|u|_{K^2_{\eta}(T)},
\end{align*}
where in the last inequality we used the condition on $\mu$.

Now, for $T\in\cT_h^{near}$, we use \eqref{Ihnear} and the condition on $\mu$ obtaining:
\begin{align*}
|u-I_h u|_{K^1_{\sigma-1}(T)} &\le |u|_{K^1_{\sigma-1}(T)}+|I_h u|_{K^1_{\sigma-1}(S'_T)}\le \bar{r}_T^{\sigma-1-\eta}|u|_{K^1_{\eta}(T)}+C_{\sigma}\|u\|_{K^2_{\sigma-1}(S_T')} \\
&\le h^{\frac{\sigma-1-\eta}{\mu}}|u|_{K^1_{\eta}(T)}+C_{\sigma} h^{\frac{\sigma-1-\eta}{\mu}}\|u\|_{K^2_{\eta}(S_T')}
\lesssim C_{\sigma} h\|u\|_{K^2_{\eta}(S_T)}.
\end{align*}
\end{proof}

Thanks to \eqref{stability}, we have the following corollary which gives optimal rates of convergence in $W_{\sigma}$ and, equivalently, in $H^1_{\sigma}$ and in $K^1_{\sigma-1}$:
\begin{coro}\label{C:pointK1sigma}
  If $\sigma$ satisfies \eqref{singularwelldefcond} and we take $\mu\le\sigma-1-\eta$:
\begin{equation}\label{pointK1sigma}
\|u-u_h\|_{W_{\sigma}(\Omega)}\lesssim h\|u\|_{K^2_{\eta}(\Omega)}.
\end{equation}
\end{coro}

\subsection{Estimates in $L^2_{\beta}$}
Now, we need to study the approximation error for \eqref{adjoint}. Taking $\beta>\frac{n}{4}-1$, $\varphi_\beta\in H^2$, and the standard Lagrange interpolator can be used. However, some effort is needed for obtaining estimates in the norm $W_{-\sigma}$, due to the negative weight.
We begin proving a local Poincar\'e inequality, deduced from \ref{T:impPoincare}:

\begin{lemma}[Local Poincar\'e inequality]\label{L:localPoincare}
  Let $T$ be such that ${\bf 0}$ is one of its vertices, and take $v\in H^1_{-\sigma}(T)$, for $\sigma$ satisfying \eqref{impPoincarecond} such that $\int_T v = 0$. Then:
 $$\|v\|_{L^2_{-\sigma}(T)}\le C h_T^{1-\sigma+\beta}\|\nabla v\|_{L^2_{-\beta}(T)}.$$
\end{lemma}
holds for every $\beta$ such that $\sigma\le \beta+1$, with a constant $C$ independent of $T$ and $v$.
\begin{proof}
  We denote $\wh{T}$ the reference element with vertices on $\{{\bf 0}\}$, and the canonical vectors $e^i$ such that $(e^i)_j = \delta_{i,j}$. Then we have an affine map $F_T: \wh{T}\to T$, $F_T(\wh{x})= x$.
  We define $\wh{v}(\wh{x}) = v(x)$.
  Also, we have that the distance $\wh{r}$ in $\wh{T}$ satisfies: $h_T\wh{r}(\wh{x}) \sim r(F_T(\wh{x}))$. Hence, we have:
  \begin{align*}
    \|v\|_{L^2_{-\sigma}(T)} = \left(\int_T v(x)^2 r(x)^{-2\sigma}dx \right)^\frac{1}{2}
    \sim h_T^{-\sigma}J_T^\frac{1}{2}\left(\int_{\wh{T}} \wh{v}(\wh{x}) \wh{r}(\wh{x})^{-2\sigma} d\wh{x}\right)^{\frac{1}{2}} = I
  \end{align*}
  where $J_T = |\det(DF_T)|.$ Now, we apply \eqref{impPoincare} on $\wh{T}$ and go back to $T$, taking into account that $\wh{\nabla}\wh{v}\sim h_T\nabla v$:
  \begin{align*}
    I &\lesssim h_T^{-\sigma}J^{\frac{1}{2}} \left(\int_{\wh{T}}|\wh{\nabla}\wt{v}|^2\wh{r}(\wh{x})^{2(1-\sigma)}d\wh{x}\right)^\frac{1}{2}
   \lesssim h_T^{-\sigma} \left(\int_T h_T^2|\nabla v|^2 h_T^{2(\sigma-1)}r(x)^{2(1-\sigma)} dx\right)^\frac{1}{2} \\
    &\lesssim \|\nabla v\|_{L^2_{1-\sigma}(T)} \lesssim h_T^{1+\beta-\sigma} \|\nabla v\|_{L^2_{-\beta}(T)}.
  \end{align*}
\end{proof}

As we commented earlier, since $\varphi_\beta\in H^2$, a standard Lagrange interpolator can be used. We take $\Pi_h(v)$ defined as in \eqref{Ihdef}, but with $a_i=v(x_i)$ for every $i$. The following result replicates Lemma \ref{L:cotasIh}:
\begin{lemma}\label{L:approxPih}
  The Lagrange interpolator $\Pi_h$ satisfies the following properties:
  \begin{equation}\label{Pihapprox}
  |u-\Pi_h v|_{H^\ell(T)}\le C h_T^{2-\ell}|v|_{H^2(T)},\quad \forall T\in \cT_h
  \end{equation}
  \begin{equation}\label{Pihstability}
  \|\Pi_h v\|_{W_{-\sigma(T)}}\lesssim h_T^{-1}\|v\|_{L^2_{-\sigma}(T)} + |v|_{H^1_{-\sigma}(T)} + h_T |v|_{H^2_{-\sigma}(T)}, \quad\forall T:\, T\cap\{{\bf 0}\}\nempty
  \end{equation}
  for every $v\in H^2_{-\sigma}$ and $\ell=0,1$.
\end{lemma}
\begin{proof}
  \eqref{Pihapprox} is a well known result. On the other hand, for \eqref{Pihstability}, we have:
  \begin{align*}
    \|\Pi_h v\|_{W_{-\sigma}(T)} \le \sum_{i:x_i\in\bar{T}} \underbrace{|a_i|}_A \underbrace{\|\phi_i\|_{W_{-\sigma}(T)}}_B.
  \end{align*}
  $B$ is easily bounded using that $|\nabla \phi_i|\sim h_T^{-1}$, and $\int_T r(x)^{-2\sigma}dx \le |T|^{\frac{1}{2}}h_T^{-\sigma}$, giving: $B\le |T|^{\frac{1}{2}}h_T^{-1-\sigma}$.
  On the other hand, taking $\wh{T}$ the reference element and using the map $F_T:\wh{T}\to T$ as in Lemma \ref{L:localPoincare}, we have:
  \begin{align*}
  |a_i|&=|v(x_i)|\le \|v\|_{L^\infty(T_i)}=\|\hat{v}\|_{L^{\infty}(\hat{T})} \le C \|\hat{v}\|_{H^{2}(\hat{T})}
  = C\bigg(\sum_{j=0}^2|\hat{v}|^2_{H^j(\hat{T})}\bigg)^{\frac{1}{2}} \\
  &= C |T|^{-\frac{1}{2}}\bigg(\sum_{j=0}^2 h_{T}^{2j}|v|_{H^j(T)}^2\bigg)^{\frac{1}{2}}
  \le C |T|^{-\frac{1}{2}}\bigg(\sum_{j=0}^2 h_{T}^{2j+2\sigma}|v|_{H^j_{-\sigma}(T)}^2\bigg)^{\frac{1}{2}}
  \end{align*}
  and joining the estimates for $A$ and $B$ the result follows.
\end{proof}

In \eqref{Ihnear}, we took advantage of the fact that $u$ belongs to a Kondratiev type space, so $\nabla u\in L^2_{\sigma}$, but $u\in L^2_{\sigma-1}$.
That is not true for $\varphi_\beta$ and thence we obtained the term $h_T^{-1}\|v\|_{L^2_{-\sigma(T)}}$ in \eqref{Pihstability}. In order to compensate this, we use that $\Pi_h$ in invariant over polynomials of degree $1$.
Let us define $P_T(v)$ the polynomial of degree $1$ such that $\int_T D^\alpha(v(x)-P_T(v)(x)) dx = 0$, for every $|\alpha| \le 1$. The following result is a natural consequence of Lemma \ref{L:localPoincare}:

\begin{lemma}\label{P_Tapprox}
  For $T\in \cT_h$ such that $T\cap \{{\bf 0}\}\nempty$, taking $\sigma$ and $\beta$ as in Lemma \ref{L:localPoincare}, the following inequalities hold:
\begin{equation}\label{localPoincare1}
  \|\nabla(v-P_T(v))\|_{L^2_{-\sigma}(T)}\le C h_T^{1-\sigma+\beta} |v|_{H^2_{-\beta}(T)}
\end{equation}
\begin{equation}\label{localPoincare2}
\|v-P_T(v)\|_{L^2_{-\sigma}(T)}\le C h_T^{2-\sigma+\beta} |v|_{H^2_{-\beta}(T)}
\end{equation}
\end{lemma}
\begin{proof}
  \eqref{localPoincare1} is given directly by Lemma \ref{L:localPoincare}.
  \eqref{localPoincare2} follows applying first Lemma \ref{L:localPoincare} with $\beta = \sigma$ and afterwards \eqref{localPoincare1}.
\end{proof}

Now, we are finally able to prove our error estimate in $L^2_{\beta}$ norm:
\begin{lemma}\label{L:adjointK1sigma}
 Let $\sigma$ satisfy \eqref{singularwelldefcond}, and $\beta$ such that $\sigma\le \beta+1$. Then, taking the grading parameter $\mu\le 1+\beta-\sigma$, we have:
 \begin{equation}
   \|\varphi_\beta-\Pi_h(\varphi_\beta)\|_{L^2_{-\sigma}(\Omega)}\le C h |\varphi_\beta|_{H^2_{-\beta}(\Omega)}.
 \end{equation}
\end{lemma}
\begin{proof}
  We prove the result elementwise. For $T$ such that ${\bf 0}\notin \bar{T}$:
\begin{align*}
  \|\nabla(\varphi_\beta-\Pi_h(\varphi_\beta))\|_{L^2_{-\sigma}(T)} &\lesssim r_T^{-\sigma}\|\nabla(\varphi_\beta-\Pi_h(\varphi_\beta))\|_{L^2(T)}\lesssim h_Tr_T^{\beta-\sigma}|\varphi_\beta|_{H^2_{-\beta}(S_T)}\\
  &\lesssim h r_T^{1-\mu+\beta-\sigma}|\varphi_\beta|_{H^2_{-\beta}(S_T)} \lesssim h |\varphi_\beta|_{H^2_{-\beta}(S_T)}.
\end{align*}
Whereas, for $T$ with ${\bf 0}$ in one of its vertices, we interpose $P_T = P_T(\varphi_\beta)$:
\begin{align*}
  \|\nabla(\varphi_\beta-\Pi_h(\varphi_\beta))\|_{L^2_{-\sigma}(T)} &\le \|\nabla(\varphi_\beta-P_T)\|_{L^2_{-\sigma}(T)} + \|\nabla\Pi_h(\varphi_\beta-P_T)\|_{L^2_{-\sigma}(T)}.
\end{align*}
Now, the second term can is bounded by \eqref{Pihstability}:
$$|\Pi_h(\varphi_\beta-P_T)|_{H^1_{-\sigma}(T)} \le h_T^{-1}\|\varphi_\beta-P_T\|_{L^2_{-\sigma}(T)}+
|\varphi_\beta-p_T|_{H^1_{-\sigma}(T)}+h_T^{1+\beta-\sigma}|\varphi|_{H^2_{-\beta}(T)}.$$
In the last term we applied a slightly adapted version of \eqref{Pihstability}, taking norm $L^2_{-\beta}$ only for the second order derivatives. Now, applying Lemma \ref{P_Tapprox}, we have:
$$\|\nabla(\varphi_\beta-\Pi_h(\varphi_\beta)\|_{L^2_{-\sigma}(T)} \le C h_T^{1-\sigma+\beta} |\varphi_\beta|_{H^2_{-\beta}(S_T)},$$
and the result follows summing up over all $T\in\cT_h$.
\end{proof}

Joining Corollary \ref{C:pointK1sigma} and Lemma \ref{L:adjointK1sigma}, we obtain an optimal order of convergence in $L^2_{\beta}$.
Since the exponent $\eta$ on the right hand side of \eqref{pointK1sigma} should be taken $\eta>\frac{n}{2}-2$, the condition $\mu\le \sigma-1-\eta$ can be reduced to $\mu< \sigma+1-\frac{n}{2}$.
Hence, in order to be able to apply both results we need $\mu<\max\{\sigma+1-\frac{n}{2},1+\beta-\sigma\}$, for any $\sigma$ such that $\frac{n}{2}-1<\sigma<\frac{n}{2}$. Consequently it is enough to take $\mu< 1+\frac{\beta}{2}-\frac{n}{4}$:

\begin{thm}\label{T:L2beta}
  For any $\beta\geq \frac{n}{4}-1$, taking $\mu<1+\frac{\beta}{2}-\frac{n}{4}$, we have:
  $$\|u-u_h\|_{L^2_{\beta}(\Omega)} = O(h^2).$$
\end{thm}

A particularly interesting result follows when taking $\beta=0$, leading to an estimate for the $L^2$ norm of the error. In $n=2$ the restriction on $\mu$ reads $\mu<\frac{1}{2}$.
In \cite{ABSV_delta} the authors propose a graded mesh with parameter $\mu=\frac{1}{2}$, and prove the suboptimal rate of convergence:
$$\|u-u_h\|_{L^2(\Omega)}\le C h^2|\log(h)|^{\frac{3}{2}}.$$
 A similar result is obtained in \cite{KW_delta}. Our numerical results are consistent with the ones exposed in \cite{ABSV_delta}, showing an order slightly worse than $2$ for $\mu=\frac{1}{2}$. However, taking $\mu<\frac{1}{2}$ the optimal order is recovered (see Table \ref{Tab:deltaR2} in Section \ref{S:experiments}).

\section{Anisotropic meshes for sources supported on segments}\label{S:segmentfem}

We treat extensively the three dimensional problem, considered in \cite{D_1d3d}. Here, we restrict ourselves to the case of $\Gamma$ being the segment \eqref{segment}. and study anisotropic graded meshes. Conclusions for isotropic meshes are derived from our calculations. Afterwards, we comment the two dimensional problem where some adjustments should be made.

 $\cT_h$ is now a graded anisotropic mesh. In $\Omega\setminus B(\Gamma,1)$, $\cT_h$ is formed by regular isotropic elemets of diameter $h$. In $B(\Gamma,1)$, on the contrary, $\cT_h$ should be graded. The idea is to grade $\cT_h$ isotropically in $B^{out}(\Gamma,1)$ and anisotropically in $C(\Gamma,1)$.
 However, according to Theorem \ref{T:anisotropic}, elements in $C(\Gamma,1)$ should be graded towards the extreme points of $\Gamma$. Hence, let us recall that $r_e(x)$ denotes the distance to the extreme points of $\Gamma$ and define
 $$B^{out}_{\tau}=\{x\in B(\Gamma,1):\, r_e(x)< \tau r(x)\},$$
 for some fixed constant $\tau>1$. In $B^{out}_{\tau}$ (where $r\sim r_e$), we define an isotropic graded mesh with:

$$h_T\sim \left\{\begin{array}{cl}
h^{\frac{1}{\mu}} & \textrm{if }\Gamma\cap\bar{T}\neq\emptyset, \\
hr_T^{1-\mu} & \textrm{if }0<r_T\le 1.
\end{array}\right.$$

 On the other hand, in $B(\Gamma,1)\setminus B^{out}_{\tau}$, we generate an anisotropic mesh.
 We recall some usual concepts:
 An element $T$ is of tensor-product type if it has one edge parallel to the $x_n$ axis, and a face ($n=3)$ or edge ($n=2$) parallel to the $x_1,x_2$ plane, or to the $x_1$ axis.
 We denote by $h_{T,j}$ ($j=1,\dots,n)$ the dimensions of the element $T$. For $n=2$, $h_{T,j}$ is the length of the edge parallel to the $x_j$ axis. For $n=3$, $h_{T,1}$ and $h_{T,2}$ are the base and the height of the face parallel to the plane $x_1$, $x_2$, and $h_{T,3}$ is the length of the edge parallel to the $x_3$ axis. $\h_T$ stands for the size vector $(h_{T,1},\dots,h_{T,n})$.
 We take $h_{T,1}\sim h_{T,2}< h_{T,3}$. We recall that $r_{T,e}$ stands for the distance from $T$ to the extreme points of $\Gamma$. We take:
 $$
 h_{T,j}\sim \left\{\begin{array}{cl}
h^{\frac{1}{\mu}} & \textrm{if } r_T = 0, \\
hr_T^{1-\mu} & \textrm{if }0<r_T\le 1,
\end{array}\right.\; (j<n);
\quad  h_{T,n}\sim \left\{\begin{array}{cl}
h^{\frac{1}{\mu}} & \textrm{if }r_{T,e} = 0, \\
h(r_{T,e})^{1-\mu} & \textrm{if }0<r_{T,e}\le 1.
\end{array}\right.$$

We proceed as in the case $m=0$, estimating first the discretizarion error of $u$ in $W_{\sigma}$, and afterwards, the interpolation error for $\varphi_\beta$ in $W_{-\sigma}$, leading to an estimate for the discretization error of $u$ in $L^2_{\beta}$. We study in detail the three dimensional problem.

\subsection{Estimate in $W_{\sigma}$ ($n=3$)}

As in the previous section, we need an interpolation operator that can be applied to functions in $K^2_{\sigma-1}$. However since we need to take into account the anisotropy of the mesh, we consider an adapted Scott-Zhang interpolator, instead of an adapted Lagrange one.

Let us recall that the interpolators of Scott-Zhang type take the form \eqref{Ihdef}
where $a_i=\mathbb{P}_{\xi_i}^k u(x_i)$. $\xi_i$ is certain non-empty set and $\mathbb{P}_{\xi_i}^k:L^2(\xi_i)\to \mathcal{P}_k(\xi_i)$ is the $L^2$ projection into the space of polynomials of degree $\le k$ on $\xi_i$.
The $S_h$ variant choose $\xi_i$ to be small edges or faces adjacent to the node $x_i$. There usually are many possible choices for $\xi_i$ fitting this criteria. Our interpolator is essentially $S_h$, but it is taken equal to zero at the segment $\Gamma$.
Specifically, we take $I_h$ of the form \eqref{Ihdef} with:
 $$a_i=0 \quad \textrm{for }x_i\in\mathcal{T}_h^{near},\quad\quad a_i = \mathbb{P}^1_{\xi_i} u (x_i) \textrm{ for }x_i\in\mathcal{T}_h^{far}.$$
 In the last case, we take $\xi_i = F$, where $F$ is a face (or edge) of the triangulation such that $F\subset\mathcal{T}_h^{far}$. In $B(\Gamma,1)\setminus B^{out}_{\tau}$, $\xi_i$ is taken parallel to the $x_1$-$x_2$ plane ($x_1$ axis). In other words, far from the singularity $I_h u = S_h u$, with a particular choice of small faces $\xi_i$.

We want to prove an analogue to Lemma \ref{L:cotasIh} for our modified Scott-Zhang interpolator. First, we recall a few useful facts, that we state with no proof. We refer the reader to \cite[Section 3]{A_Anisotropic} for details.
Let us observe that in every $T\in\cT_h^{far}$ the weight $r_T^{\sigma}$ is essentially constant, so the weighted space is equivalent to the unweighted one. Consequently, we have that $(\mathbb{P}_{\xi_i} u) (x_i) = \int_{\xi_i} u\psi_i,$
where $\psi_i\in \cP_1(\xi_i)$ is such that $\int_{\xi_i} \psi_i \phi_j = \delta_{i,j} \quad \forall i,j,$ and $\|\psi_i\|_\infty \sim |\xi_i|^{-1}$.

We also need the following trace theorem, that holds both for $n=2$ and $n=3$:
\begin{lemma}\label{L:tracesL1}
Let $\xi$ be an edge $(n=2)$ or face $(n=3)$ of an element $T$, $\ell\in\N$ and $p\geq 1$. Then, for every $v\in H^{\ell}(T)$ we have that $v$ has a trace in $e$ in the sense of $L^1$ and that:
$$\|v\|_{L^1(\xi)}\le C|\xi||T|^{-\frac{1}{p}}\sum_{|\alpha|\le \ell}\h^{\alpha}_T\|D^{\alpha} v\|_{L^p(T)},$$
where $|\xi|$ is the $(n-1)-$dimensional measure of $\xi$.
\end{lemma}
\begin{proof}
The result follows by changing variables to the reference edge/face $\hat{\xi}$ corresponding to $\xi$, in the reference element $\hat{T}$, applying the trace theorem there and going back to the original element $T$.
\end{proof}

Now we can prove the following analogue to Lemma \ref{L:cotasIh}.

\begin{lemma}\label{L:approxIh}
Let $I_h$ be the adapted Scott-Zhang operator. Then, for $T\in \cT_h^{far}$ we have:
\begin{equation}\label{approxIhanis}
|u-I_h u|_{H^1(T)} \le C \sum_{|\alpha| = 1} \h_T^{\alpha}|D^{\alpha}u|_{H^{1}(S_T)}.
\end{equation}
When $T\subset B^{out}_{\tau}$, $\h_T$ should be replaced by $h_T$. For $T\in\cT_h^{near}$, $T\subset B^{out}_{\tau}$:
\begin{equation}\label{stabilityIhanis}
|I_h u|_{K^1_{\sigma-1}(T)} \le C \Big\{\|u\|_{L^2_{\sigma-1}(S'_T)}+\sum_{j=1}^2 \|\partial_{x_j} u\|_{L^2_{\sigma}(S'_T)}+ h_{T,3} \|\partial_{x_3} u\|_{L^2_{\sigma-1}(S'_T)}\Big\}
\end{equation}
\end{lemma}
\begin{proof}
The first inequality follows from the fact that $I_h$ is identical to $S_h$ on $\mathcal{T}_h^{far}$. See, for example \cite{A_Anisotropic}.
For \eqref{stabilityIhanis}, let us begin observing that if every node of $T$ is inside $\cT_h^{near}$, the left hand side vanishes, so there is nothing to prove. Hence, the result should be proven for every $T\in\cT_{near}$ such that there is some $T'\in\cT_h^{far}$ with $\bar{T}\cap\bar{T'}\nempty$. Let us denote $D_T = \{i:\; T_i\in\cT_h^{far},\, \bar{T}\cap\bar{T_i}\nempty\}$.
We have that:
\begin{align*}
|I_h u|_{K^1_{\sigma-1}(T)}\le \sum_{i\in D_T}\,\underbrace{\int_{\xi_i}|u\psi_i|}_{A}\,\underbrace{|\phi_i|_{K^1_{\sigma-1}(T)}}_{B}
\end{align*}
For $A$ we use that $\|\psi_i\|_{L^\infty(\xi_i)}\sim|\xi_i|^{-1}$, and apply Lemma \ref{L:tracesL1}:
\begin{align*}
A&\le C |\xi_i|^{-1}\|u\|_{L^1(\xi_i)} \le C|T_i|^{-\frac{1}{2}}\sum_{|\alpha|\le 1}\h_T^{\alpha}\|D^{\alpha} u\|_{L^2(T_i)},
\end{align*}
where $T_i$ is any element, $T_i\in\cT_h^{far}$ such that $\xi_i\subset \bar{T_i}$. We can continue:
\begin{align*}
\sum_{|\alpha|\le 1}\h_T^{\alpha} &\|D^{\alpha}u\|_{L^2(T_i)} \\
&\le r_{T_i}^{1-\sigma}\|u\|_{L^2_{\sigma-1}(T_i)}+\sum_{j=1}^2 h_{T_i,j} r_{T_i}^{-\sigma}\|\partial_{x_j} u\|_{L^2_{\sigma}(T_i)}+h_{T_i,3}
 r_{T_i}^{1-\sigma}\|\partial_{x_3} u\|_{L^2_{\sigma-1}(T_i)} \\
&\le h^{\frac{1-\sigma}{\mu}}\big\{\|u\|_{L^2_{\sigma-1}(T_i)}+\sum_{j=1}^2 \|\partial_{x_j} u\|_{L^2_{\sigma}(T_i)}+ h_{T_i,3} \|\partial_{x_3} u\|_{L^2_{\sigma-1}(T_i)}\big\}
\end{align*}

On the other hand, for $B$, considering each derivative of $\phi_i$, we have:

\begin{align*}
\|\partial_{x_j} \phi_i\|_{L^2_{\sigma}}\le C h_{T_i,j}^{-1} \bar{r}_{T_i}^{\sigma} |T_i|^{\frac{1}{2}} \le C h_{T_i,j}^{-1} h_{T,1}^{{\sigma}} |T_i|^{\frac{1}{2}}
\le C h_{T,1}^{\sigma-1}|T_i|^{\frac{1}{2}}\le C h^{\frac{\sigma-1}{\mu}}|T_i|^{\frac{1}{2}},
\end{align*}
where we used that $h_{T,j}\geq h_{T,1}$ for every $j$.
Finally, we can combine the estimations for $A$ and $B$, obtaining \eqref{stabilityIhanis}.
\end{proof}

We are now able to prove the approximation result for anisotropic meshes:

\begin{thm}\label{T:approxanis}
Let $\cT_h$ be a graded anisotropic mesh as defined previously, and $u_h$ the finite element solution of problem \eqref{weakproblem}. Then if $\mu<\sigma$, we have that:
\begin{equation}\label{anisoapprox}
\|u- u_h\|_{W_{\sigma}(\Omega)} = O(h).
\end{equation}
\end{thm}
 \begin{proof}
 Thanks to \eqref{stability}, we only need to estimate $\|u-I_h u\|_{W_{\sigma}(\Omega)}$. As usual, we proceed element-wise. Let us take $T\in \cT_h^{far}$, and assume $T\subset B(\Gamma,1)\setminus B^{out}_{\tau}$. The case $T\subset B^{out}_{\tau}$ is easier, since no anosotropy should be considered.
 We use extensive that for $T\in \cT_{far}$, $\bar{r}_T\sim r_T$ and that $r_T\sim r_{T'}$ and $h_{T,j}\sim h_{T',j}$ for $j=1,2,3$ and any $T'\in S_T$. We denote $\alpha'$ a multiindex with $\alpha_n=0$, to denote derivatives with respect to the first variables:
 \begin{align*}
\|u &-I_h u\|_{W_{\sigma}(T)} \lesssim \bar{r}_T^{\sigma} \|u-I_h\|_{H^1(T)} \lesssim
    \bar{r}_T^{\sigma} \sum_{|\alpha|=1}\h^{\alpha}|D^{\alpha}u|_{H^1(S_T)}\\
&\lesssim \bar{r}_T^{\sigma}\bigg\{h_{T,1}\sum_{|\alpha'|=2}\|D^{\alpha'}u\|_{L^2(S_T)} +
   h_{T,3}  \sum_{|\alpha'|=1}\|\partial_{x_3}D^{\alpha'}u\|_{L^2(S_T)} +
    h_{T,3} \|\partial_{x_3}^2 u\|_{L^2(S_T)} \bigg\}\\
&\lesssim \bar{r}_T^{\sigma} \bigg\{h_{T,1}r_T^{-\eta-2} \sum_{|\alpha'|=2}
  \|D^{\alpha'} u\|_{L^2_{\eta+2}(S_T)}+h_{T,3}r_T^{-\eta-1} \sum_{|\alpha'|=1}\|\partial_{x_3}D^{\alpha'} u\|_{L^2_{\eta+1}(S_T)}\\
  &\quad\quad\quad +h_{T,3}r_T^{-\eta}\|\partial^2_{x_3} u\|_{L^2_{\eta}(S_T)}\bigg\} \\
&\lesssim h \bar{r}_T^{\sigma-1-\eta-\mu} \bigg\{\sum_{|\alpha'|=2} \|D^{\alpha'}
   u\|_{L^2_{\eta+2}(S_T)}+r_{T,e}\Big(\frac{r_T}{r_{T,e}}\Big)^\mu  \sum_{|\alpha'|=1}\|\partial_{x_3}D^{\alpha'} u\|_{L^2_{\eta+1}(S_T)}\\
   &\quad\quad\quad+r_{T,e}r_T \Big(\frac{r_T}{r_{T,e}}\Big)^\mu\|\partial^2_{x_3} u\|_{L^2_{\eta}(S_T)}\bigg\}
 \end{align*}
 Now, assuming $\mu\le\sigma-1-\eta$, and the fact that $r_T\le r_{T,e}$ we conclude:
 \begin{align*}
\quad&\lesssim h \bigg\{\sum_{|\alpha'|=2} \|D^{\alpha'} u\|_{L^2_{\eta+2}(S_T)}+r_{T,e} \sum_{|\alpha'|=1}\|\partial_{x_3}D^{\alpha'} u\|_{L^2_{\eta+1}(S_T)} +r_{T,e}^2\|\partial^2_{x_3} u\|_{L^2_{\eta}(S_T)}\bigg\},
 \end{align*}
  It is clear that the case $T\in\cT_h^{far}$, $T\subset B^{out}_{\tau}$ can be solved in the same way taking the diameter of $T$, $h_T$ instead of $h_{T,j}$ for every $j$, and increasing the weight on the norms of the derivatives with respect to $x_3$.
 Finally, let us consider $T\in\cT_h^{near}$. Again, the interesting case is given by the anisotropic elementes, $T\subset B(\Gamma,1)\setminus B^{out}_{\tau}$ and $r_{T,e}>0$. Using that $h_{T,1}\sim h^{\frac{1}{\mu}}$ and applying \eqref{stabilityIhanis} we have:
 \begin{align*}
 \|u&-I_h u\|_{W_{\sigma}(T)} \le |u|_{K^{1}_{\sigma-1}(T)}+|I_h u|_{K^{1}_{\sigma-1}(T)} \\
   &\lesssim \bar{r}_{T}^{\sigma-\eta-1}|u|_{K^1_{\eta}(T)}+
     \|u\|_{L^2_{\sigma-1}(S'_T)}+\sum_{j=1}^2 \|\partial_{x_j} u\|_{L^2_{\sigma}(S'_T)}+ h_{T,3} \|\partial_{x_3} u\|_{L^2_{\sigma-1}(S'_T)}\\
 &\lesssim \bar{r}_T^{\sigma-\eta-1}\Big\{|u|_{K^1_\eta(T)}+\|u\|_{L^2_{\eta}(S'_T)}+\sum_{j=1}^2 \|\partial_{x_j} u\|_{L^2_{\eta+1}(S'_T)}+ h r_{T,e}^{1-\mu} \|\partial_{x_3} u\|_{L^2_{\eta}(S'_T)}\Big\}
 \end{align*}
  But, $\bar{r}_{T}^{\sigma-\eta-1}\sim h^{\frac{\sigma-\eta-1}{\mu}}$. So assuming, once again, $\mu\le \sigma-\eta-1$, we have that $\bar{r}_T^{\sigma-\eta-1}\le h$.
The result follows summing up over all the elements, taking into account that the overlapping of the patches $S_T$ is finite, and that the factors given by powers of $r_{T,e}$ are enough to compensate the lack of weight in the norms of the derivatives with respect to $x_3$ (see Theorem \ref{T:anisotropic}).

The condition $\mu\le\sigma-\eta-1$ can be used for any $\eta>-1$ (Theorem \ref{T:norms}), so the condition on $\mu$ reduces to: $\mu<\sigma$.
Observe that the last estimation has a term $r_{T,e}^{1-\mu}\|\partial_{x_3}u\|_{L^2_\eta(S_T')}$. Theorem \ref{T:anisotropic} indicates that a factor $r_{T,e}^\frac{1}{2}$ is necessary, so we are induced to think that a condition $\mu\le\frac{1}{2}$ should be stated.
This is not true, though. Indeed, if we take $\mu>\frac{1}{2}$ we can write, in the left factor of the last inequality $\bar{r}_T^{\sigma-\eta-1}=\bar{r}_T^{\sigma-\eta-\mu-\frac{1}{2}}\bar{r}_T^{\mu-\frac{1}{2}}$.
Hence, we obtain in the last term: $hr_{T,e}^{\frac{1}{2}}(\bar{r}_T/r_{T,e})^{\mu-\frac{1}{2}}\le hr_{T,e}^\frac{1}{2}$,
so in order to preserve the order $1$ given by the $h$ factor, we just need $\sigma-\eta-\mu-\frac{1}{2}\geq 0$, which is equivalent to: $\mu \le \sigma-\eta-\frac{1}{2}$. Once again, since $\eta>-1$, this reduces to: $\mu\le \sigma+\frac{1}{2}$, which is true since $\mu\le\sigma$.
 \end{proof}

\begin{rmk}
  This result is consistent with the one proved in \cite{D_1d3d}, where isotropic graded meshes are considered, and a condition $\mu<\sigma$ is provided to guarantee an order $h$ for the $W_{\sigma}$ norm of the error.
\end{rmk}

\subsection{Estimate in $L^2_{\beta}$ ($n=3$)}
Here again we need to produce estimates for the $W_{-\sigma}$ norm of the error of the adjoint problem \eqref{adjoint}. We follow the ideas of Section \ref{S:pointfem}, though a little more technical problems arise due to the anisotropy of the mesh.

Since $\varphi_\beta\in H^2$, we can use the standard Scott-Zhang interpolator $S_h$. We need the following anisotropic version of the local Poincar\'e inequality given in Lemma \ref{L:localPoincare}:

\begin{lemma}[Local Anisotropic Poincar\'e ineaquality]\label{L:localanisoPoincare}
Consider $T$ such that $T\cap\Gamma\nempty$, (one of $T$'s edges lie on $\Gamma$), and take $v\in H^1_{-\sigma}$, for $\sigma$ satisfying \eqref{impPoincarecond}.  Then, if $\int_T v = 0$,  the inequality:
$$\|v\|_{L^2_{-\sigma}(T)}\le C h_{T,1}^{-\sigma}\sum_{j}h_{T,j}\|\partial_{x_j} v\|_{L^2(T)},$$
holds for every $\beta$ such that $\sigma\le \beta+1$, with a constant $C$ independent of $T$ and $v$.
\end{lemma}
\begin{proof}
  The result follows in the same line than Lemma \ref{L:localPoincare}, taking into account that the transformed distance $\wh{r}$ satisfies $h_{T,1}\wh{r}(\wh{x})=r(F_T(\wh{x}))$.
\end{proof}

We also need a weighted stability result for $S_h$, in order to handle the negative exponent $-\sigma$ near $\Gamma$. The proof uses some technical tricks that are usual for this kind of interpolator.

\begin{lemma}
Let $T\in\cT_h$ such that one of its edges lie on the segment $\Gamma$, and take $v\in W_{-\sigma}\cap H^2$. We denote $\nabla'v$ the gradient of $v$ with respect the first variables (excluding $x_n$),. Analogously, $\nabla^{2'}$ stands for the derivatives of $v$ of order two, with respect to the first variables. Then for $j=1,2$:
\begin{equation}\label{stabilitySh1}
  \|\partial_{x_j}S_h v\|_{L^2_{-\sigma}(T)}\le \|\partial_{x_j}v\|_{L^2_{-\sigma}(T)}+h_{T,1}^{1-\sigma}\|\nabla^{2'}v\|_{L^2(T)} +
   h_{T,3}h_{T,1}^{-\sigma}\|\partial_{x_3}\nabla'v\|_{L^2(T)}
\end{equation}
And:
\begin{equation}\label{stabilitySh2}
  \|\partial_{x_3}S_h v\|_{L^2_{-\sigma}(T)}\le \|\partial_{x_3}v\|_{L^2_{-\sigma}}+h_{T,1}h_{T,3}^{-1}\|\nabla' v\|_{L^2_{-\sigma}(T)}.
\end{equation}
\end{lemma}
\begin{proof}
  We have:
  $$\|\partial_{x_j} S_h v\|_{L^2_{-\sigma}(T)} \le \sum_{i:x_i\in T}\underbrace{\bigg|\int_{\xi_i}v\psi_i\bigg|}_{A}\underbrace{\|\partial_{x_j}\phi_i\|_{L^2_{-\sigma}}}_{B}.$$
  $B$ is bounded as usual: $B\le h_{T,j}^{-1}|T|^\frac{1}{2}h_{T,1}^{-\sigma}$. For $A$, let us begin considering the simpler case $j=3$. Then, we can take a constant $\omega$ and observe that $\partial_{x_3}S_h v = \partial_{x_3}(S_h(v-\omega))$, which gives, applying Lemma \ref{L:tracesL1}:
  \begin{align*}
    A &\le |\xi_i|^{-1}\int_{\xi_i}|v-\omega| \le |T|^{-\frac{1}{2}}\sum_{|\alpha|\le 1}\vec{h}_T^{\alpha}\|D^{\alpha} (v-\omega)\|_{L^2(T)} \\
     &\le|T|^{-\frac{1}{2}}\big\{\|v-\omega\|_{L^2(T)}+h_{T,1}\|\nabla' v\|_{L^2(T)}+h_{T,3}\|\partial_{x_3}v\|_{L^2(T)}\big\}
  \end{align*}
  Taking $\omega = \frac{1}{|T|}\int_T v$, and applying Lemma \ref{L:localanisoPoincare}:
  \begin{align*}
    A&\le |T|^{-\frac{1}{2}}\big\{h_{T,1}\|\nabla' v\|_{L^2(T)}+h_{T,3}\|\partial_{x_3}v\|_{L^2(T)}\big\}
  \end{align*}
  \eqref{stabilitySh2} follows taking $\|\partial_{x_i}v\|_{L^2(T)}\le h_{T,1}^{\sigma}\|\partial_{x_i}v\|_{L^2_{-\sigma}}$ and multiplying $A\cdot B$.
  Observe that the same argument holds for any derivative in the isotropic elements of the mesh.

  For \eqref{stabilitySh1} we proceed in a similar way, taking $\omega=\omega(x_3)$ a polynomial of degree $1$. Moreover, since $T$ is an element of tensor product type we have that the faces $\xi_i$ that participate in the definition of $S_h$ on $T$ belong to two parallel planes, orthogonals to the $x_3$ axis.
  Hence, we can take two sets $\bar{\xi}_1$ and $\bar{\xi}_2$ such that $\xi_i\subset \bar{\xi}_k$  and $|\xi_i|\sim|\bar{\xi}_k|$, for some $k=1,2$  and for every $j$ such that $x_j\in \bar{T}$.
  Taking $\omega$ such that $\omega|_{\bar{\xi}_k} = |\bar{\xi}_k|^{-1}\int_{\bar{\xi}_k} v$,   applying the classical Poincar\'e inequality in $L^1(\bar{\xi}_k)$ (see \cite[Theorem 3.2]{AD_Poincare}), and taking into account that diam$(\bar{\xi}_k)\sim h_{T,1}$:
  \begin{align*}
    A &\le |\xi_i|^{-1}\int_{\xi_i}(v-\omega)\le |\xi_i|^{-1}\int_{\bar{\xi}_k}(v-\omega) \le |\xi_i|^{-1} h_{T,1}\|\nabla' v\|_{L^1(\bar{\xi}_k)}.
  \end{align*}
  Applying Lemma \ref{L:tracesL1}:
  \begin{align*}
    A&\le |\xi_i|^{-1}h_{T,1}|\bar{\xi}_k||T|^{-\frac{1}{2}}\sum_{|\alpha|\le 1} \vec{h_T}^{\alpha}\|D^{\alpha}\nabla' v\|_{L^2(T)}\\
    &\le h_{T,1}|T|^{-\frac{1}{2}} \big\{ \|\nabla' v\|_{L^2(T)}+h_{T,1}\|\nabla^{2'}v\|_{L^2(T)}+h_{T,3}\|\partial_{x_3}\nabla' v\|_{L^2(T)} \big\},
  \end{align*}
  and \eqref{stabilitySh1} follows taking $\|\nabla' v\|_{L^2(T)}\le h_{T,1}^{\sigma} \|\nabla' v\|_{L^2_{-\sigma}(T)}$ and multiplying $A\cdot B$.
\end{proof}

Finally, we are able to prove the approximation result:

\begin{thm}\label{T:approxanisadjoint}
Given $\varphi\in W_{-\sigma}\cap H^2_{-\beta}(\Omega)$ with $\beta\geq\sigma$:
$$\|\varphi-S_h\varphi\|_{W_{-\sigma}} \lesssim h |\varphi|_{H^2_{-\beta}},$$
for every graduation parameter $\mu$.
\end{thm}
\begin{proof}
  For an element $T$ such that $\bar{T}\cap\Gamma = \emptyset$, the result follows applying the approximation property of $S_h$ (see \eqref{approxIhanis}):
  \begin{align*}
    |\varphi &-S_h \varphi|_{H^1_{-\sigma}(T)} = r_T^{-\sigma}|\varphi-S_h|_{H^1(T)}
    \lesssim r_T^{-\sigma}\sum_{|\alpha|=1}\vec{h}_T^{\alpha}|D^{\alpha}\varphi|_{H^1(S_T)}
  \end{align*}
  \vspace{-0.5cm}
  \begin{equation}\label{badequation}
    \le h r_T^{1-\mu+\beta-\sigma}\sum_{|\alpha'|=1} |D^{\alpha'}\varphi|_{H^1_{-\beta}(S_T)}+h r_T^{\beta-\sigma}r_{T,e}^{1-\mu}|\partial_{x_3}\varphi|_{H^1_{-\beta}(S_T)}
  \end{equation}
  So to obtain an estimate $O(h)$, we need $\beta\geq\sigma$.

  For $T$ such that one of its edges lie on $T$ we begin interposing the polynomial $P_T=P_T(\varphi)$ such that $\int D^{\alpha}(\varphi-P_T) = 0$ for $|\alpha|\le 1$:
  \begin{align*}
    \|\nabla (\varphi - S_h\varphi)\|_{L^2_{-\sigma}(T)}\le \|\nabla(\varphi-P_T)\|_{L^2_{-\sigma}(T)}+\|\nabla S_h (\varphi-P_T)\|_{L^2_{-\sigma}(T)} = I + II.
  \end{align*}
  As we will soon see, it is enough to estimate $II$. We separate the estimation in two cases, depending on the devative considered. Applying \eqref{stabilitySh1}, we have:
  \begin{align*}
    \|\partial_{x_1} S_h (\varphi-P_T)\|_{L^2_{-\sigma}(T)}&\le
     \|\partial_{x_1}(\varphi-P_T)\|_{L^2_{-\sigma}(T)}+ h_{T,1}^{1-\sigma}\|\nabla' \varphi\|_{L^2(T)} \\
     &+ h_{T,3}h_{T,1}^{-\sigma}\|\partial_{x_3}\nabla'\varphi\|_{L^2(T)}.
  \end{align*}
  The first term on the right is a part of $I$, and can be estimated using Lemma \ref{L:localanisoPoincare} given exactly the rest of the right member. So we have:
  \begin{align*}
    \|\partial_{x_1} S_h (\varphi-P_T)\|_{L^2_{-\sigma}(T)}&\lesssim
      h_{T,1}^{1-\sigma}\|\nabla' \varphi\|_{L^2(T)} + h_{T,3}h_{T,1}^{-\sigma}\|\partial_{x_3}\nabla'\varphi\|_{L^2(T)}\\
      &\lesssim
      h_{T,1}^{1+\beta-\sigma}\|\nabla'\varphi\|_{L^2_{-\beta}(T)}
      + h_{T,3}h_{T,1}^{\beta-\sigma}\|\partial_{x_3}\nabla'\varphi\|_{L_{-\beta}^2(T)}
  \end{align*}
 The same holds for $\partial_{x_2}(S_h(\varphi-P_T))$. On the other hand,
 \begin{align*}
   \|\partial_{x_3} S_h (\varphi-P_T)\|_{L^2_{-\sigma}(T)}&\le
    \|\partial_{x_3}(\varphi-P_T)\|_{L^2_{-\sigma}(T)}+ h_{T,1}h_{T,3}^{-1}\|\nabla' (\varphi-P_T)\|_{L^2_{-\sigma}(T)}
 \end{align*}
 Again, the first term on the right hand side is part of $I$. Both terms can bounded applying Lemma \ref{L:localanisoPoincare}. We continue:
  \begin{align*}
   \|\partial_{x_3} S_h (\varphi &-P_T)\|_{L^2_{-\sigma}(T)}  \le h_{T,1}h_{T,3}^{-1}\Big(h_{T,1}^{1-\sigma}\|\nabla' \varphi\|_{L^2(T)}+h_{T,1}^{-\sigma}+h_{T,3}\|\nabla'\partial_{x_3}\varphi\|_{L^2(T)}\Big) \\
    &\hspace{3cm} +h_{T,1}^{1-\sigma}\|\nabla'\partial_{x_3}\varphi\|_{L^2(T)}+h_{T,1}^{-\sigma}h_{T,3}\|\partial_{x_3}^2\varphi\|_{L^2(T)}\\
    &\le h_{T,1}^{1+\beta-\sigma} \Big( \|\nabla'\partial_{x_3}\varphi\|_{L^2_{-\beta}(T)} +  \|\nabla'\partial_{x_3}\varphi\|_{L^2_{\beta}(T)} \Big) + h_{T,3}h_{T,1}^{\beta-\sigma}\|\partial_{x_3}^2 \varphi\|_{L^2_{-\beta}(T)}
  \end{align*}
  In order to obtain an estimate $O(h)$ we need $\mu\le \beta+1-\sigma$ which is always true since, $\beta\geq\sigma$.
\end{proof}

Joining Theorem \ref{T:approxanis} and Theorem \ref{T:approxanisadjoint}, we prove:
\begin{coro}[Estimate in $L^2_{\beta}$]
  Taking $\mu<\beta$ and $\beta>0$, we have that:
  $$\|u-u_h\|_{L^2_{\beta}}=O(h^2).$$
 \end{coro}
 \begin{proof}
   For any $\beta>0$, we can pick some $\sigma=\beta$, and $u\in W_{\sigma}$, so the estimates for $\|u-u_h\|_{W_{\sigma}}$ and $\|\varphi_\beta-\varphi_h\|_{W_{-\sigma}}$ hold.
 \end{proof}

 \begin{rmk}
It is important to notice that the previous result does not give estimates for the $L^2$ norm of the error. In this sense, Theorem \ref{T:approxanisadjoint} fails in the simplest estimate given in \eqref{badequation} for elements \emph{far} from $\Gamma$.
This is a consequence of the lack of a regularity result for $\varphi_\beta$ that read as $|\varphi_\beta|_{H^2_{-\sigma}(\Omega)}\le \|u-u_h\|_{L^2_\beta}(\Omega),$
for some $\sigma>\beta$.

However, if we consider \emph{isotropic} graded meshes, where $h_{T,3}$ is graded as $h_{T,1}$, we have that \eqref{badequation} is bounded by $r_T^{1-\mu+\beta-\sigma}h|\varphi|_{H^2_{-\beta}}$, so the condition $\mu<1+\beta-\sigma$ is enough to obtain an estimate $O(h)$. It is easy to check that the same condition works for elements $T$ such that $T\cap\Gamma\nempty$. Hence, we have the following corollary.
 \end{rmk}

\begin{coro}
If $\cT_h$ is an isotropic graded mesh with parameter $\mu<\frac{1+\beta}{2}$, we have that
$\|u-u_h\|_{L^2_{\beta}(\Omega)}= O(h^2).$
\end{coro}
\begin{proof}
 We have that if $\mu<\sigma$ then, $\|u-u_h\|_{W_{\sigma}}=O(h)$, and if $\mu<1+\beta-\sigma$, $\|\varphi_\beta-S_h\varphi_\beta\|_{W_{-\sigma}}\le C h \|u-u_h\|_{L^2_\beta}$,
 so we need to take $\mu<\max\{\sigma,1+\beta-\sigma\}$, for any $0<\sigma<1$, so $\mu<(\beta+1)/2$ is enough to obtain $O(h^2)$.
\end{proof}


\subsection{Two dimensional problem}

The case $n=2$ is slightly different. On the one hand, $u$ does not belong to $K^2_{\sigma-1}$: the gradient of $u$ is as smooth as $u$ itself. This does not allow a stability estimate like \eqref{stabilityIhanis}.
However, since $u\in H^1_0$ there is no need to truncate the interpolation operator as we did in the three dimensional problem: we can use the standard Scott-Zhang interpolator $S_h$. On the other hand, using the truncated interpolator $I_h$ we were able to treat elements in $\cT_h^{near}$ as elements in $\cT_h^{far}$ where the weight can be pulled out or pushed in the norms as needed. While considering $S_h$ for the adjoint problem, we took advantage of the fact that $\varphi\in H^2$. None of these strategies are possible when using $S_h$ for $u$ in $n=2$.

We prove only a stability estimate analogue to \eqref{stabilityIhanis}, particularly for $\partial_{x_1}S_h v$. The rest of the analysis is similar to the case $n=3$.

\begin{lemma}
  Let $T\in\cT_h^{near}$, $T\subset B(\Gamma,1)\setminus B^{out}_\tau$. Given $u$ the solution of \eqref{problem},  the following stability estimates hold:
  \begin{equation}\label{stabilitySh2dx1}
    \|\partial_{x_1}S_h u\|_{L^2_{\sigma}(T)} \lesssim \|\partial_{x_1} u\|_{L^2_{\sigma}(S_T)}+\|\partial_{x_1}^2 u\|_{L^2(S_T)_{\sigma+1}} + h_{T,2}\|\partial^2_{x_2 x_1}u\|_{L^2_{\sigma}(S_T)}
  \end{equation}
  For $\partial_{x_2}S_h u$, \eqref{stabilitySh2} holds.
\end{lemma}
\begin{proof}
  We proceed as in the proof of \eqref{stabilitySh1}: the three sets $\xi_i$ corresponding to the nodes $x_i\in\bar{T}$ are contained in two sets $\bar{\xi}_1$ and $\bar{\xi}_2$ paralells to the $x_1$ axis, with $|\xi_i|\sim |\bar{\xi}_k|$.
  We take a polynomial $w=w(x_2)$ such that $w|_{\bar{\xi}_k}=|\bar{\xi}_k|^{-1}\int_{\bar{\xi}_k} u$, and obtain:
  \begin{align*}
    \|\partial_{x_1} &S_h u\|_{L^2_{\sigma}(T)} = \|\partial_{x_1}S_h (u-w)\|_{L^2_{\sigma}(T)} \le \sum_{x_i\in T} \left|\int_{\xi_i}(u-w)\psi_i\right|\|\partial_{x_1}\phi_i\|_{L^2_{\sigma}(T)}\\
    &\lesssim \sum_{k=1,2}|\bar{\xi}_k|^{-1}h_{T,1}^{-1}\|r^{\sigma}\|_{L^2(T)} \int_{\bar{\xi}_k}|u-w| =I
  \end{align*}
  Now, we apply the improved Poincar\'e inequality (see for example \cite{BS_Poincare}):
  $$\|u-w\|_{L^1(\bar{\xi}_{k})}\lesssim \|\partial_{x_1}u\cdot d\|_{L^1(\bar{\xi}_k)} \le \|x_1\partial_{x_1}u \|_{L^1(\bar{\xi}_k)},$$
  where $d$ represents here the distance to the boundary of the $1$-dimensional set $\bar{\xi}_k$.  Hence, we can continue applying Lemma \ref{L:tracesL1} and H\"older inequality:
  \begin{align*}
     \int_{\bar{\xi}_k} &|u-w| \le  \|x_1\partial_{x_1}u \|_{L^1(\bar{\xi}_k)} \le |\bar{\xi}_k||T_i|^{-1}\sum_{|\alpha|\le 1}\h_T^\alpha\|D^{\alpha}(x_1\partial_{x_1}u )\|_{L^1(S_T)} \\
     &\le |\bar{\xi}_k||T|^{-1}
          \big\{ \|x_1 \partial_{x_1}u \|_{L^1(S_T)} + h_{T,1}\|x_1 \partial^2_{x_1}u\|_{L^1(S_T)} + h_{T,2}\|\partial^2_{x_2 x_1}u\|_{L^1(S_T)}\big\}\\
          &\le |\bar{\xi}_k||T|^{-1}\|r^{-\sigma}\|_{L^2} \big\{\|x_1 \partial_{x_1}u \|_{L^2_{\sigma}} + h_{T,1}\|x_1 \partial^2_{x_1}u\|_{L^2_{\sigma}} +
           h_{T,2}\|x_1\partial^2_{x_2 x_1}u\|_{L^2_{\sigma}}\big\} \\
     &\le |\bar{\xi}_k||T|^{-1}\|r^{-\sigma}\|_{L^2} \big\{\|x_1 \partial_{x_1}u \|_{L^2_{\sigma}} + h_{T,1}\|x_1 \partial^2_{x_1}u\|_{L^2_{\sigma}} +
      h_{T,2}h_{T,1}\|\partial^2_{x_2 x_1}u\|_{L^2_{\sigma}}\big\}
  \end{align*}
 where all the norms are taken in $S_T$. Now, we observe that $\|r^\eta\|_{L^2(S_T)}\le |T|^\frac{1}{2}h_{T,1}^\eta$, for any $\eta$ satisfying \eqref{A2cond}, so using this for $\eta=\sigma$ and $\eta=-\sigma$ the result follows.
\end{proof}

With this result we can prove:

\begin{thm}\label{T:approxanis2D}
Let $\cT_h$ be a graded anisotropic mesh, and $u_h$ the finite element solution of problem \eqref{weakproblem}. Then if $\mu<\sigma+\frac{1}{2}$, we have that:
\begin{equation}\label{anisoapprox2D}
\|u- u_h\|_{W_{\sigma}(\Omega)} = O(h).
\end{equation}
\end{thm}
\begin{proof}
We consider only the case where the calculation is different than the one used in $n=3$. Take $T\in\cT_h^{near}$, then:
\begin{align*}
  \|\partial_{x_1} &(u-S_h u)\|_{L^2_{\sigma}(T)}\le \|\partial_{x_1}u\|_{L^2_{\sigma}(T)} + \|\partial_{x_1}S_h u\|_{L^2_\sigma(T)} \\
  &\lesssim \|\partial_{x_1} v\|_{L^2_{\sigma}(S_T)}+\|\partial_{x_1}^2 v\|_{L^2(S_T)_{\sigma+1}} + h_{T,2}\|\partial^2_{x_2 x_1}v\|_{L^2_{\sigma}(S_T)} \\
  &\le \bar{r}_{T}^{\sigma-\eta} \big\{\|\partial_{x_1} v\|_{L^2_{\eta}(S_T)}+\|\partial_{x_1}^2 v\|_{L^2(S_T)_{\eta+1}} + h_{T,2}\|\partial^2_{x_2 x_1}v\|_{L^2_{\eta}(S_T)}\big\}.
\end{align*}
Now, $\bar{r}_T \sim h_{T,1} = h^\frac{1}{\mu}$, so in order to obtain an $O(h)$ estimate, we need $\mu\le \sigma-\eta$. The same condition is obtained for $T\in\cT_h^{far}$ and for $\partial_{x_2}(u-S_hu)$ when $T\in\cT_h^{near}$. Since the restriction $\eta>-\frac{1}{2}$ holds, the result follows.
\end{proof}

The analysis for the adjoint problem is exactly as in the three dimensional case. Once again, for anisotropic meshes the estimate for $|\varphi-S_h\varphi|_{H^1_{-\sigma}}$ requieres $\beta\geq\sigma$, so we obtain:
\begin{thm}
  Taking $\mu<\beta+\frac{1}{2}$, we have that:
  $$\|u-u_h\|_{L^2_{\beta}}=O(h^2).$$
\end{thm}

However, for isotropic meshes $\beta$ can be taken $\beta<\sigma$ as long as the restriction $\mu<1+\beta-\sigma$ is satisfied, which combined with the condition $\mu<\sigma+\frac{1}{2}$ of Theorem \ref{T:approxanis2D} gives:
\begin{thm}
  For isotropic meshes, taking $\mu<\frac{3}{4}+\frac{\beta}{2}$, we have:
 $$\|u-u_h\|_{L^2_{\beta}}=O(h^2),$$
\end{thm}

\section{Numerical experiments}\label{S:experiments}
In this section we present our numerical results. We implemented a solver in Matlab, following closely the compact implementation proposed in \cite{ACF_50lines}.

\subsection{Point delta}
We solve Problem \eqref{problem} taking $\Omega = B(0,1)$.
The exact solution is:
$$u(r) = \left\{\begin{array}{cr}
-\frac{\log(r)}{2\pi} & n=2\\
\frac{1}{4\pi}\Big(\frac{1}{r}-1\Big) & n=3,
\end{array}\right.$$
where $r=\|x\|$.

The graded meshes were obtained in two different ways. The first strategy is the one used in \cite{ABSV_delta}: we built a regular mesh of size $H$ with a set of points $Q$ and then scale this points taking: $p=q\|q\|^{\frac{1-\mu}{\mu}}$. In this way, for each $\mu$, we have two meshes: a uniform one, and a graded one, both with the same number of nodes. This allows a direct comparison between the results obtained using graded and not-graded meshes.
On the other hand, since meshes built as explained above have elements that are much larger in the radial component than in the angular ones, we also tested our results with graded meshes \emph{by construction}, i.e.: built directly by taking a set of radii $r_i$ such that $r_1\sim h^{\frac{1}{\mu}}$ and $r_{i+1}-r_i\sim hr_i^{1-\mu}$, and defining, on $\partial B(0,r_i)$ a set of points at a distance $\sim hr_i^{1-\mu}$ from each other. This method produces more regular meshes. The comparison with uniform meshes is no longer direct, but we observe that similar magnitudes of the error can be obtained with less points. Figure \ref{F:mallasR2} shows the three kind of meshes in $\R^2$: uniform, graded by re-scaling and graded by construction.

\begin{figure}[!ht]
\begin{center}
\includegraphics[width=0.28\textwidth]{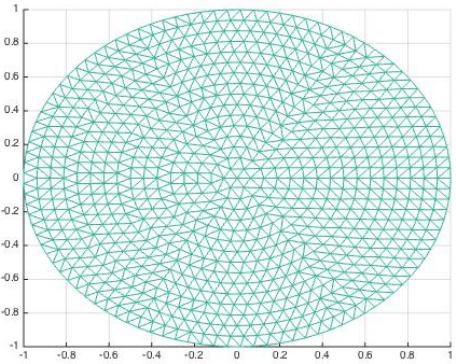}
\includegraphics[width=0.28\textwidth]{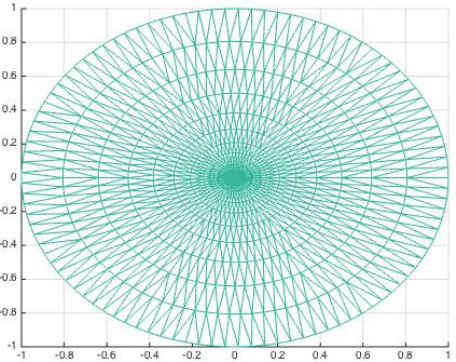}
\includegraphics[width=0.28\textwidth]{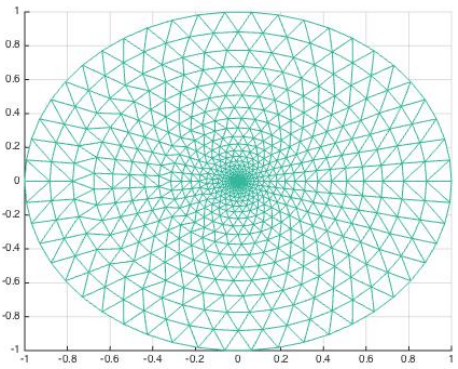}
\caption{Three meshes: uniform (left), graded by re-scaling (center), both with $N=856$, and graded by construction (right), with $N=730$, in $\R^2$ with graduation parameter $\mu=0.3$.}\label{F:mallasR2}
\end{center}
\end{figure}

Table \ref{Tab:deltaR2} shows results for $n=2$ on meshes graded by re-scaling whereas Table \ref{Tab:deltaR3} shows results for $n=3$ on meshes graded by construction. In both tables $\mu$ is the grading parameter and $h$ the step used to build the mesh. $h$ is reported only for illustrative purposes, since the order of convergence is estimated using the number of nodes, $N$.
The norms $\|u-u_h\|_{L^2(\Omega)}$ and $\|u-u_h\|_{L^2_{\beta}(\Omega)}$ are approximated through order $3$ quadrature rules. Finally, $e.o.c.$ stands for the estimated order of convergence.

{\small
\begin{longtable}{|c|c|c|c|c|c|c|c|c|c|}
  \hline
 \multirow{3}{*}{$h$} & \multirow{3}{*}{$N$} & \multicolumn{2}{c|}{$\mu=0.4$} &  \multicolumn{2}{c|}{$\mu=0.5$} &    \multicolumn{2}{c|}{$\mu=0.6$} & \multicolumn{2}{c|}{$\mu=1$} \\
 \cline{3-10}
 & & $L^2$ & $L^2_{\beta}$ & $L^2$ & $L^2_{\beta}$ & $L^2$ & $L^2_{\beta}$ & $L^2$  & $L^2_{\beta}$ \\
 & & $\times 10^{-3}$ & $\times 10^{-3}$ & $\times 10^{-3}$ & $\times 10^{-3}$ &
     $\times 10^{-3}$ & $\times 10^{-3}$ & $\times 10^{-3}$ & $\times 10^{-3}$ \\
 \hline
 $2^{-4}$ &   $856$ & $5.802$ & $4.076$ & $5.187$ & $2.766$ & $5.686$ & $2.243$ & $9.639$ & $3.783$ \\
 $2^{-5}$ &  $3319$ & $1.147$ & $1.022$ & $1.411$ & $0.693$ & $1.812$ & $0.569$ & $4.822$ & $1.443$ \\
 $2^{-6}$ & $13070$ & $0.371$ & $0.256$ & $0.379$ & $0.173$ & $0.575$ & $0.144$ & $2.411$ & $0.548$ \\
 $2^{-7}$ & $51875$ & $0.093$ & $0.064$ & $0.101$ & $0.043$ & $0.182$ & $0.036$ & $1.206$ & $0.208$ \\
 \hline\hline
 \multicolumn{2}{|c|}{e.o.c} & $2.013$ & $2.025$ & $1.92$ & $2.025$ & $1.678$ & $1.919$ & $1.013$ & $1.414$ \\
 \hline
 \caption{Error in $L^2$ and in $L^2_\beta$ for $\beta = 0.4$, and estimated order of convergence for $-\Delta u = \delta$, in $n=2$, for different mesh graduations, and meshes graded by re-scaling a uniform mesh.}
 \label{Tab:deltaR2}
\end{longtable}
}

These numerical experiments are consistent with our predictions. In both cases an order $\sim 2$ is obtained in $L^2$, when $\mu<1-\frac{n}{4}$. In the critical case $\mu_0=1-\frac{n}{4}$, we observe a loss of order that increases with $\mu$, leading to an order $\sim 1$ when $\mu = 2\mu_0$. For weighted norms $L^2_\beta$ with $\beta>0$, the order is $\sim 2$ for a wider range of values of $\mu$.

{\small
\begin{longtable}{|c|c|c|c|c|c|c|c|c|c|}
  \hline
 \multirow{3}{*}{$h$}  &  \multicolumn{3}{c|}{$\mu=0.18$} & \multicolumn{3}{c|}{$\mu=0.25$} &
                          \multicolumn{3}{c|}{$\mu=0.5$} \\
 \cline{2-10}
     & \multirow{2}{*}{$N$} & $L^2$ & $L^2_{\beta}$
     & \multirow{2}{*}{$N$} & $L^2$ & $L^2_{\beta}$
     & \multirow{2}{*}{$N$} & $L^2$ & $L^2_{\beta}$ \\
     & & $\times 10^{-3}$ & $\times 10^{-4}$ & & $\times 10^{-3}$ & $\times 10^{-4}$ & &
     $\times 10^{-3}$ & $\times 10^{-4}$  \\
 \hline
 $2^{-3}$            &  $18159$ & $1.066$ & $3.330$ & $13483$ & $1.402$ & $3.346$ & $6853$ & $5.448$ & $5.331$ \\
 $2^{-\frac{10}{3}}$ & $33388$ & $0.707$ & $2.197$ & $23753$ & $0.979$ & $2.261$ & $12084$ & $4.450$ & $3.592$ \\
 $2^{-\frac{11}{3}}$ & $62413$ & $0.463$ & $1.411$ & $45004$ & $0.654$ & $1.147$ & $21736$ & $3.595$ & $2.547$ \\
 $2^{-4}$            & $118854$ & $0.299$ & $0.916$ & $89943$& $0.430$ & $0.980$ & $42552$ & $2.840$ & $1.747$ \\
 \hline\hline
 e.o.c & & $2.03$ & $2.06$  & & $1.90$ & $1.98$ & & $1.07$ & $1.84$ \\
 \hline
 \caption{Error in $L^2$ and in $L^2_\beta$ for $\beta = 0.7$, and estimated order of convergence for $-\Delta u = \delta$, in $n=3$, for different graduations, on meshes graded by construction.}
 \label{Tab:deltaR3}
\end{longtable}
}

We observe, naturally, that in meshes graded \emph{by construction}, for the same mesh parameter $h$, the number of nodes $N$ decreases when $\mu$ increases. Similar results are obtained when using meshes graded by construction for $n=2$ or meshes graded by rescaling for $n=3$.

\subsection{Segment singularity}
 As explained above, anisotropic meshes where built with elements of tensor product type on $B(\Gamma,1)\setminus B^{out}_\tau$. Figure \ref{F:segment2d} illustrates the difference between isotropic and anisotropic graded mesh. It shows the domain $\Omega=B(\Gamma,1)\subset\R^2$ where  $\Gamma$ is given by \eqref{segment} with $L=1$. In both meshes the parameters chosen are $h=2^{-3}$ and $\mu = 0.4$. In the anisotropic mesh $\tau=0.8$.

 \begin{figure}[!h]
 \begin{center}
 \includegraphics[width=0.35\textwidth]{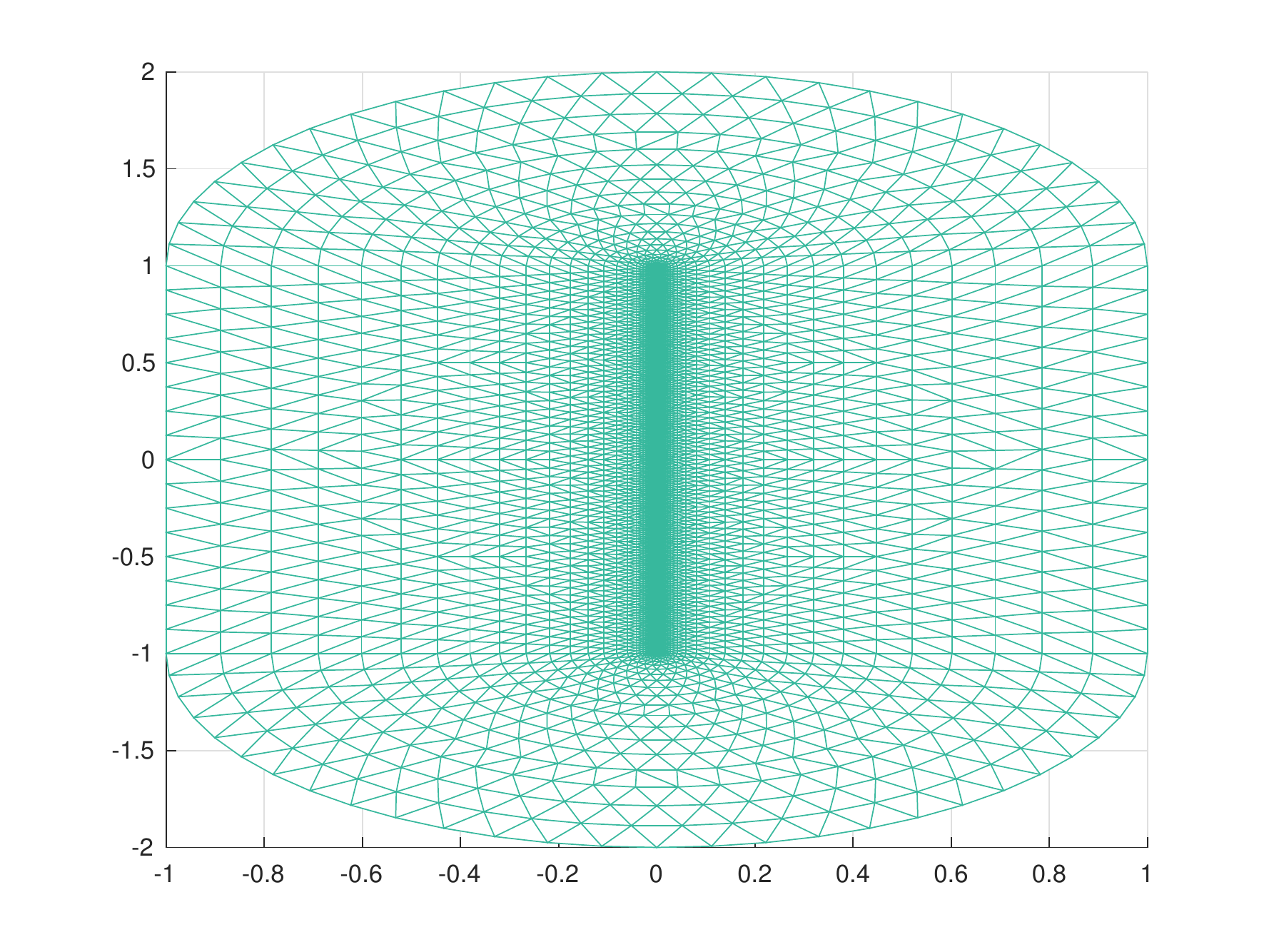}
 \includegraphics[width=0.35\textwidth]{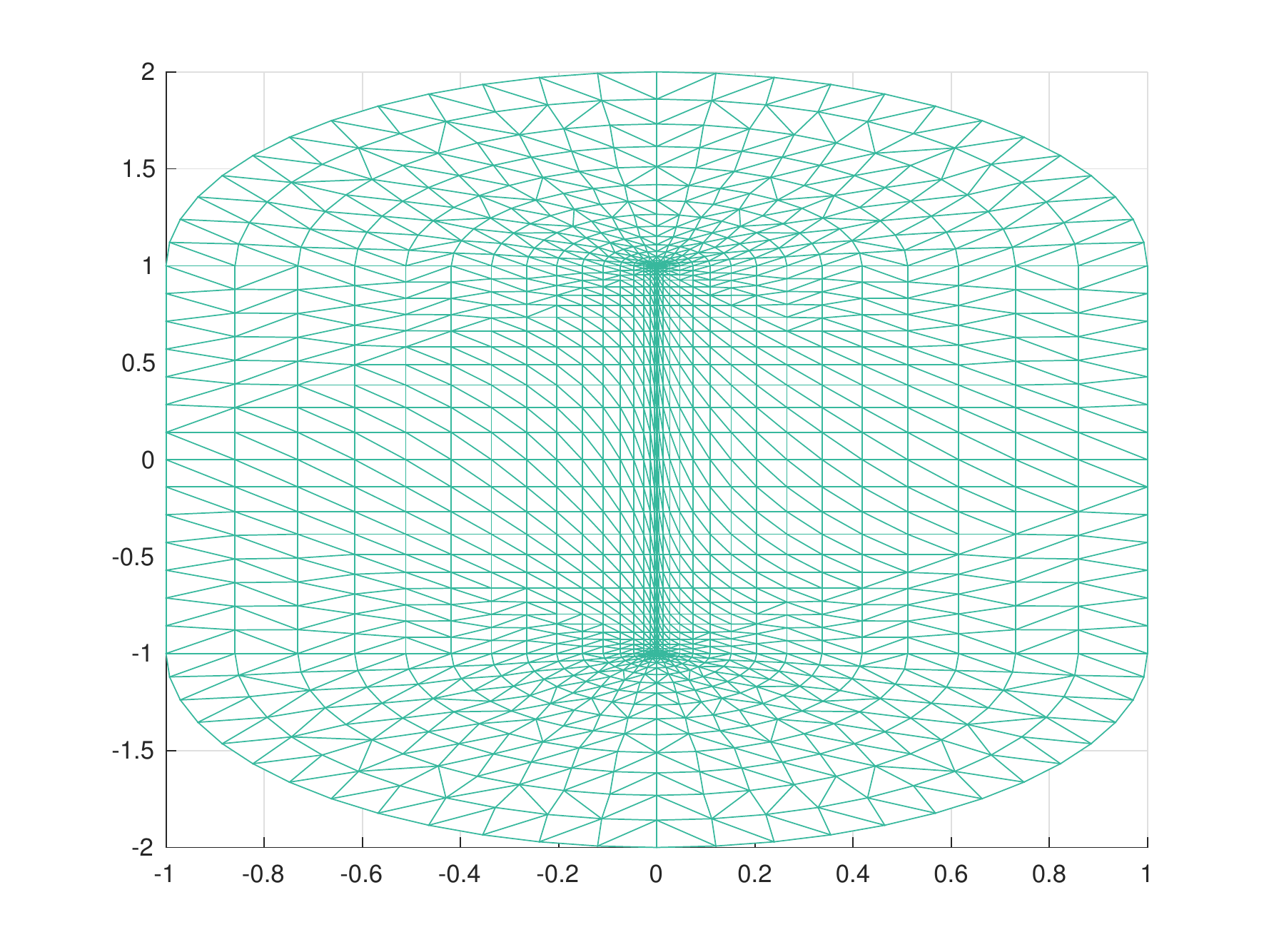}
 \caption{Isotropic and anisotropic graded meshes, with $h=2^{-3}$, $\mu=0.4$.}\label{F:segment2d}
 \end{center}
 \end{figure}

  In $\R^3$ meshes were built in the same way. We show approximation results only in $\R^3$. There, the fundamental solution for our problem (taking $L=1$ and $\hat{\gamma}=\frac{1}{2}$) is:

 $$\int_{-1}^1 \frac{1}{4\pi \sqrt{x^2+y^2+(z-t)^2}} dt = \frac{1}{4\pi} \log\bigg(\frac{\sqrt{x^2+y^2+(z-1)^2}+1-z}{\sqrt{x^2+y^2+(z+1)^2}-1-z}\bigg).$$

 The argument of the logarithm is equal to $3$ on the boundary of the solid ellipsoid:
$$\Omega: \frac{x^2}{3}+\frac{y^2}{3}+\frac{z^2}{4} \le 1,$$
so we have that:
$$G(x,y,z) = \frac{1}{4\pi} \log\bigg(\frac{\sqrt{x^2+y^2+(z-1)^2}+1-z}{3\big(\sqrt{x^2+y^2+(z+1)^2}-1-z\big)}\bigg).$$
is the exact solution of problem \eqref{problem} on $\Omega$.

{\small
\begin{longtable}{|c|c|c|c|c|c|c|c|c|c|}
  \hline
 \multirow{3}{*}{$h$}  &  \multicolumn{3}{c|}{$\mu=0.4$} & \multicolumn{3}{c|}{$\mu=0.5$} &
                          \multicolumn{3}{c|}{$\mu=1$} \\
 \cline{2-10}
     & \multirow{2}{*}{$N$} & $L^2$ & $L^2_{\beta}$
     & \multirow{2}{*}{$N$} & $L^2$ & $L^2_{\beta}$
     & \multirow{2}{*}{$N$} & $L^2$ & $L^2_{\beta}$ \\
     & & $\times 10^{-2}$ & $\times 10^{-2}$ & & $\times 10^{-2}$ & $\times 10^{-2}$ & &
     $\times 10^{-2}$ & $\times 10^{-2}$  \\
 \hline
 $0.4$ &  $1445$ & $1.04$ & $0.60$ & $1149$ & $1.39$ & $0.73$ & $693$ & $3.09$ & $1.92$ \\
 $0.2$ &  $9060$ & $0.34$ & $0.20$ & $6894$ & $0.48$ & $0.32$ & $3902$ & $2.06$ & $1.00$ \\
 $0.1$ & $58679$ & $0.10$ & $0.06$ & $49999$ & $0.13$ & $0.05$ & $33663$ & $1.01$ & $0.33$ \\
 \hline\hline
 e.o.c(N) & & $1.68$ & $1.72$  & & $1.71$ & $1.87$ & & $0.84$ & $1.27$ \\
 \hline
 e.o.c(h) & & $1.89$ & $1.92$  & & $1.89$ & $1.97$ & & $0.85$ & $1.77$ \\
 \hline
 \caption{Error in $L^2$ and in $L^2_\beta$ for $\beta = 0.4$, and estimated order of convergence for $-\Delta u = \delta_{\Gamma}$, in $n=3$, for different graduations, on meshes graded by construction.}
 \label{Tab:seg3daniso}
\end{longtable}
}

Table \ref{Tab:seg3daniso} shows results for anisotropic meshes whereas Table \ref{Tab:seg3diso} shows results for the same norm and grading parameters on isotropic meshes. We observe that the order of convergence is similar for both the unweighted and the weighted $L^2$ norm. We show to estimations of the orders of convergence: one computed in terms of the number of nodes ($e.o.c.(N)$), and one computed in terms of the mesh parameter $h$ ($e.o.c.(h)$).

{\small
\begin{longtable}{|c|c|c|c|c|c|c|c|c|c|}
  \hline
 \multirow{3}{*}{$h$}  &  \multicolumn{3}{c|}{$\mu=0.4$} & \multicolumn{3}{c|}{$\mu=0.5$} &
                          \multicolumn{3}{c|}{$\mu=1$} \\
 \cline{2-10}
     & \multirow{2}{*}{$N$} & $L^2$ & $L^2_{\beta}$
     & \multirow{2}{*}{$N$} & $L^2$ & $L^2_{\beta}$
     & \multirow{2}{*}{$N$} & $L^2$ & $L^2_{\beta}$ \\
     & & $\times 10^{-2}$ & $\times 10^{-2}$ & & $\times 10^{-2}$ & $\times 10^{-2}$ & &
     $\times 10^{-2}$ & $\times 10^{-2}$  \\
 \hline
 $0.4$ &  $1340$ & $1.16$ & $0.68$ & $1010$ & $1.46$ & $0.79$ & $566$ & $3.40$ & $2.11$ \\
 $0.2$ &  $11085$ & $0.36$ & $0.21$ & $7027$ & $0.49$ & $0.23$ & $3429$ & $2.30$ & $1.10$ \\
 $0.1$ & $87220$ & $0.10$ & $0.06$ & $56018$ & $0.14$ & $0.06$ & $2914$ & $1.06$ & $0.35$ \\
 \hline\hline
 e.o.c(N) & & $1.75$ & $1.77$  & & $1.70$ & $1.82$ & & $0.84$ & $1.70$ \\
 \hline
 e.o.c(h) & & $1.74$ & $1.77$  & & $1.89$ & $1.82$ & & $0.87$ & $1.76$ \\
 \hline
 \caption{Error in $L^2$ and in $L^2_\beta$ for $\beta = 0.4$, and estimated order of convergence for $-\Delta u = \delta_{\Gamma}$, in $n=3$, for different graduations, on meshes graded by construction.}
 \label{Tab:seg3diso}
\end{longtable}
}

Finally, Table \ref{Tab:isovsaniso3d} compares the number of points ($N$) and elements ($NT$) in isotropic and anisotropic meshes. It is important to notice that we are considering meshes on a domain much larger that $B(\Gamma,1)$. Consequently the number of nodes and tetrahedra in $\Omega\setminus B(\Gamma,1)$ is essentially the same in anisotropic than in anisotropic meshes. However, it is possible to observe, even for relatively large values of $h$ that the number of tetrahedra increases much more on isotropic meshes.

{\small
\begin{longtable}{|c|c|c|c|c|c|}
\hline
$\mu$ & h & $N_{iso}$ & $N_{aniso}$ & $NT_{iso}$ & $NT_{aniso}$ \\*
\hline
\multirow{3}{*}{$0.4$}  & $0.4$  &    $1340$ &    $1445$ &   $7474$  &  $8093$\\*
                        & $0.2$  &   $11085$ &   $9060$ &  $64958$  & $5293$\\*
                        & $0.1$  &  $87220$ &   $58679$ & $519688$  & $348466$\\
\hline
\multirow{3}{*}{$0.5$}  & $0.4$  &   $1010$ &    $1149$ &   $5503$  & $6364$\\*
                        & $0.2$  &   $7027$ &   $6894$ &  $40547$  & $39863$\\*
                        & $0.1$  &  $56018$ &   $49999$ &  $331453$  & $295455$\\
\hline
\caption{Number of nodes and number of tetrahedra in isotropic and anisotropic meshes for different values of $\mu$ and $h$.}
\label{Tab:isovsaniso3d}
\end{longtable}}

Figure \ref{F:sol_seg_3d} shows a solution in $\Omega$ for an anisotropic and an isotropic mesh.
\begin{figure}[!h]
\begin{center}
 \includegraphics[width=0.35\textwidth]{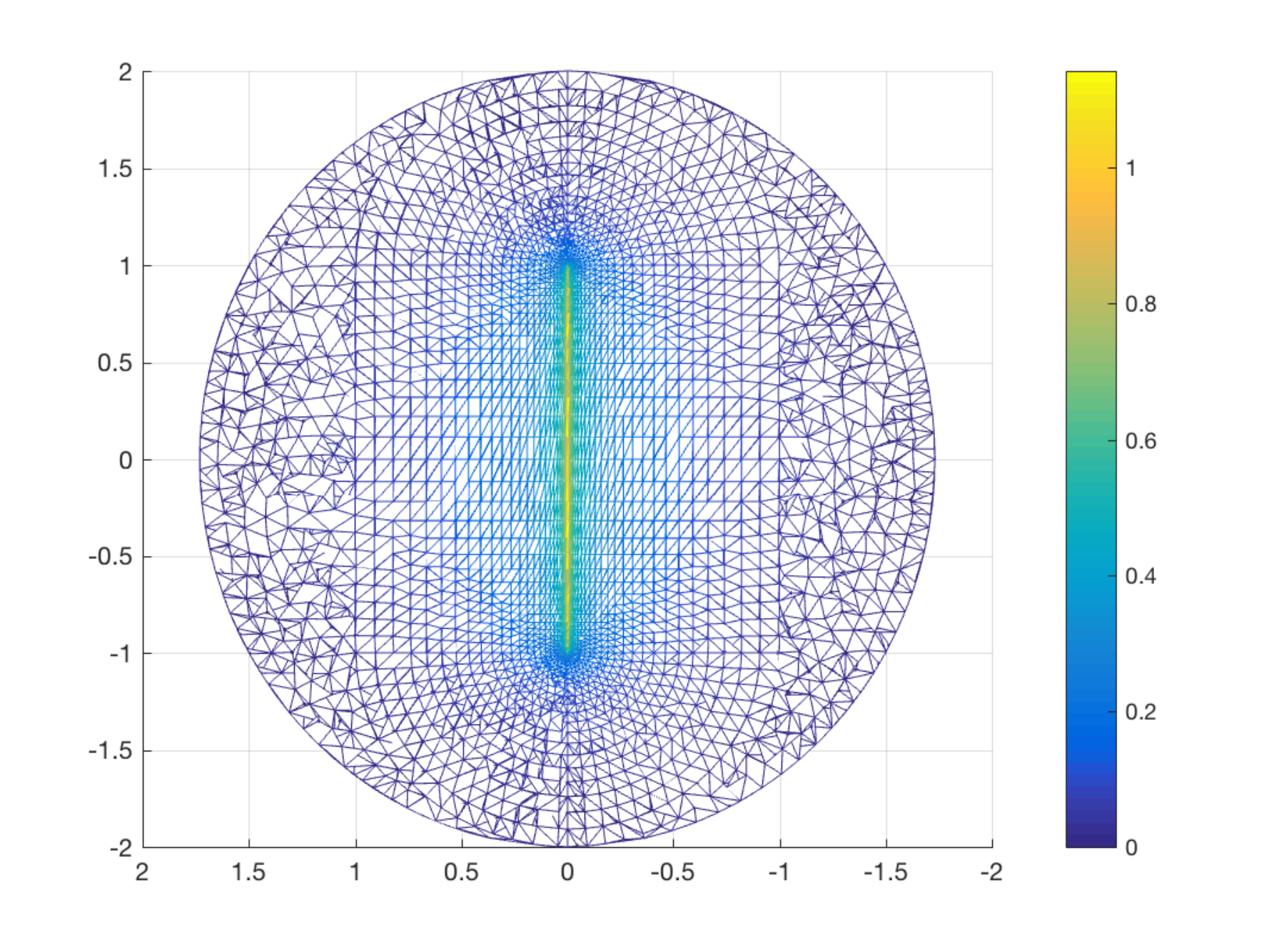}
 \includegraphics[width=0.35\textwidth]{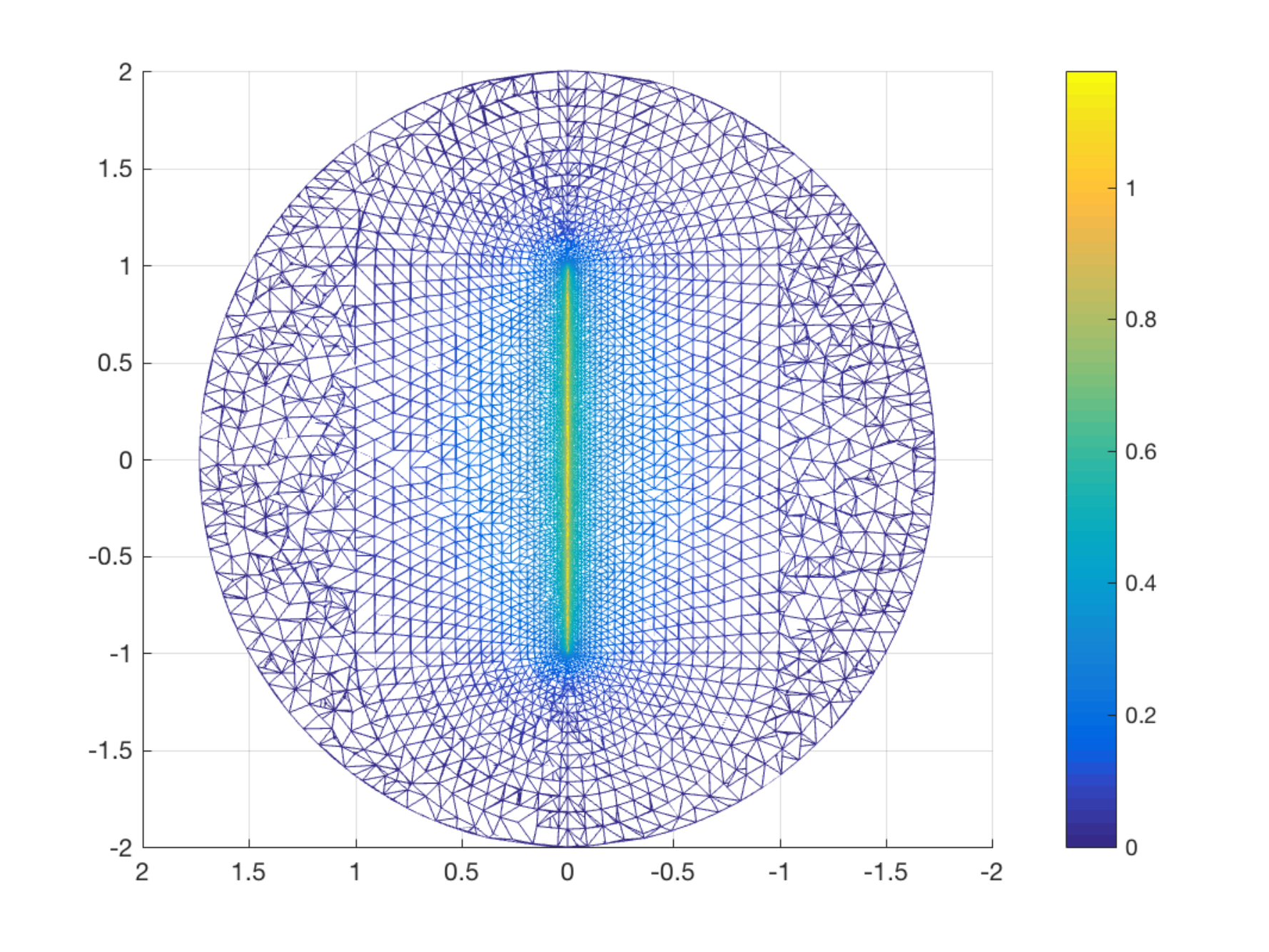}
 \caption{Solution on triangulation of $\Omega$ by an anisotropic mesh (left) and by an isotropic mesh (right), with $\mu = 0.4$ and $h=0.4$}\label{F:sol_seg_3d}
 \end{center}
 \end{figure}

\FloatBarrier

\section*{Acknowledgements} I want to thank Gabriel Acosta and Ricardo Dur\'an for their comments and suggestions.

\bibliographystyle{plain}
\bibliography{bibdoc}

\end{document}